\documentclass{amsart}

\usepackage{latexsym,amsmath,amstext,amsthm}
\usepackage[english]{babel}
\usepackage[T1]{fontenc}
\usepackage[dvips]{graphicx}
\usepackage{delarray}

\theoremstyle{plain}

\newtheorem*{thm*}{Theorem}
\newtheorem*{rem*}{Remark}

\newtheorem*{Almgren2}{Section 2.7(2) in \cite{Frederick J.
Almgren}}
\newtheorem*{Almgren3}{Section 2.9 in \cite{Frederick J.
Almgren}}
\newtheorem*{Almgren4}{Section 2.10 in \cite{Frederick J.
Almgren}}
\newtheorem*{Dir1}{Theorem 2.16 in
\cite{Frederick J. Almgren}}
\newtheorem*{Almgren5}{Theorem A.1.6(17) in \cite{Frederick J.
Almgren}}

\newtheorem*{thmmorrey}{Theorem 3.5.2 in \cite{Morrey}}
\newtheorem*{lemwillem}{Lemma 23.7 in \cite{Willem}}
\newtheorem*{Almgrenm}{Theorem 2.12
in \cite{Frederick J. Almgren}}

\newtheorem{thm}{Theorem}

\newtheorem{lem}[thm]{Lemma}
\newtheorem{prop}[thm]{Proposition}
\newtheorem{definition}[thm]{Definition}

\newcommand{\D}{\displaystyle}
\newcommand{\DF}[2]{\frac{\D#1}{\D#2}}
\def\Xint#1{\mathchoice
{\XXint\displaystyle\textstyle{#1}}%
{\XXint\textstyle\scriptstyle{#1}}%
{\XXint\scriptstyle\scriptscriptstyle{#1}}%
{\XXint\scriptscriptstyle\scriptscriptstyle{#1}}%
\!\int}
\def\XXint#1#2#3{{\setbox0=\hbox{$#1{#2#3}{\int}$}
\vcenter{\hbox{$#2#3$}}\kern-.5\wd0}}

\def\dashint{\Xint-}

\title{Regularity of Dirichlet nearly minimizing multiple-valued functions}

\author{Jordan Goblet and Wei Zhu}

\address{Jordan Goblet\\
Département de mathématique\\
Université catholique de Louvain\\
Chemin du cyclotron 2\\
1348 Louvain-La-Neuve (Belgium)} \email{goblet@math.ucl.ac.be}
\address{Wei Zhu\\
Department of Mathematics\\
Rice University\\
6100 Main Street\\
Houston, TX 77005 (U.S.A)} \email{weizhu@rice.edu}

\date{\today}

\begin{document}

\maketitle
\begin{abstract}
In this paper, we extend the related notions of Dirichlet
quasiminimizer, $\omega-$minimizer and almost minimizer to the
framework of multiple-valued functions in the sense of Almgren and
prove Hölder regularity results. We also give examples of those
minimizers with various branch sets.
\end{abstract}

\section*{Introduction}

In \cite{Frederick J. Almgren} F. J. Almgren, Jr. has developed a
theory of functions defined on a bounded open subset $\Omega$ of
$\mathbb{R}^m$ with smooth boundary and taking values in the space
${\bf Q}_{Q}(\mathbb{R}^n)$ of unordered $Q$ tuples of points of
$\mathbb{R}^{n}$ (see \cite{Frederick J. Almgren2} for a summary).
In terms of Dirichlet's integral naturally defined for appropriate
functions $\Omega\rightarrow {\bf Q}_{Q}(\mathbb{R}^n)$ he
obtained the following existence and regularity results:
\begin{thm*}
For each appropriate function $v:\partial \Omega\rightarrow {\bf
Q}_{Q}(\mathbb{R}^n)$ there exists a function $u:\Omega\rightarrow
{\bf Q}_{Q}(\mathbb{R}^n)$ having boundary values $v$ and of least
Dirichlet integral among such functions. Furthermore, each such
minimizer $u$ is Hölder continuous.
\end{thm*}
\noindent Almgren introduced the machinery of Dirichlet minimizing
multiple-valued functions in order to study the partial regularity
of mass-minimizing integral currents. He proved that any
$m$-dimensional mass-minimizing integral current is regular except
on a set of Hausdorff dimension at most $m-2$. In the present
paper, we will define and study the regularity of Dirichlet
 quasiminimizer as well as $\omega-$minimizer and almost
minimizer in the setting of multiple-valued functions.

The notion of quasiminimizer, introduced by M. Giaquinta and E.
Giusti \cite{giaqgiusti2}, embodies the notion of minimizer of
variational integrals
\begin{equation}\label{functional}
J(u,\Omega)=\int_\Omega f(x,u,\text{D}u)\;dx
\end{equation}
where $\Omega$ is an open subset of $\mathbb{R}^m$,
$u:\Omega\rightarrow\mathbb{R}^n$,
D$u=\left\{\text{D}_{\alpha}u^i\right\}$, $\alpha=1,\hdots,m$,
$i=1,\hdots,n$ and
$f(x,u,p):\Omega\times\mathbb{R}^{n}\times\mathbb{R}^{mn}\rightarrow\mathbb{R}$
, but it is substantially more general. By definition, a
quasiminimizer $u$ minimizes the functional $J$ only up to a
multiplicative constant, that is, there exists $K\geq 1$ such that
$$J(u,V)\leq K\;J(u+\phi,V)$$
 for each subdomain $V$ that is compactly contained in $\Omega$ and for
each $\phi\in C_0^{\infty}(V)$. In the scalar-valued case and
under suitable conditions on $f$, Hölder continuity can be derived
directly from the quasiminimizing property. For more details of
the statement above, as well as for more information on the
properties of quasiminimizers, we refer to \cite{Giaquintabook},
\cite{giaqgiusti2} and \cite{giusti}.\par

An alternative notion called $\omega$-minimizer is considered in a
paper by G. Anzellotti \cite{G. Anzellotti}. A function $u\in
W^{1,2}(\Omega,\mathbb{R}^n)$ is a $\omega$-minimizer, where
$\omega:(0,R_0)\rightarrow \mathbb{R}$ is a non-negative function,
for a functional $J$ of the type (\ref{functional}) if one has
$$J(u,U^m(x,r))\leq (1+\omega(r))\; J(u+\phi, U^m(x,r))$$
for any ball $U^m(x,r)\subset \Omega$ with $r<R_0$ and for all
functions $\phi\in W^{1,2}_0(U^m(x,r),\mathbb{R}^n)$. If the
functional $J$ is the Dirichlet integral and $\lim_{r\rightarrow
0}\omega(r)=0$ then Anzellotti proves that $u$ is locally Hölder
continuous. In addition the first derivatives of $u$ are locally
Hölder continuous, with exponent $\gamma$, in $\Omega$ provided
$0\leq \omega(r)\leq cr^{2\gamma}$.\par

Motivated by the question of boundary regularity of Dirichlet minimizing
multiple-valued functions (more precisely, subtracting a suitable extension
of boundary value from the original minimizer gives an almost minimizer), we also consider
$(c,\alpha)$-almost minimizer in the following sense:
$$J(u,U^m(x,r))\leq J(v,U^m(x,r))+c\;r^{m-2+\alpha}$$
for any ball $U^m(x,r)\subset \Omega$ and for all functions $v\in
W^{1,2}(U^m(x,r),\mathbb{R}^n)$ such that $u=v$ on $\partial
U^m(x,r)$. Readers are also referred to \cite{hk} for the
discussion of almost minimizer in the study of ferromagnetism in
$\mathbb{R}^3$.\par

This work is organized as follows. In Section \ref{quasireg} we
define the notion of Dirichlet  quasiminimizing multiple-valued
functions and Hölder regularity results are derived from two
Almgren's estimates. We also construct a one-dimensional
quasiminimizer with a fractal branch set. In Section \ref{Wmini}
we prove a similar result to Anzellotti's Theorem for
$\omega$-minimizing multiple-valued functions and we give an
example of $\omega-$minimizer with a unique branch point. In
Section \ref{almost} we prove a Hölder regularity theorem for
 Dirichlet almost minimizing multiple-valued functions and construct
 a one-dimensional
 almost minimizer with a fractal branch set.

\section{Preliminaries}\label{Pre}
In our terminology and notations we shall follow scrupulously
\cite{Frederick J. Almgren} and \cite{H. Federer}. Throughout the
whole text $m,n$ and $Q$ will be positive integers. We denote by
$U^m(x,r)$ and $B^m(x,r)$ the open and closed balls in
$\mathbb{R}^m$ with center $x$ and radius $r$.

\subsection{The metric space ${\bf Q}_Q(\mathbb{R}^n)$} The space of
unordered $Q$ tuples of points in $\mathbb{R}^n$ is
$${\bf
Q}_Q(\mathbb{R}^n)=\left\{\sum_{i=1}^{Q}[[u_i]]:u_1,\hdots,u_Q\in\mathbb{R}^n\right\},$$
where $[[u_i]]$ is the Dirac measure at $u_i$. This space is
 equipped with the metric $\mathcal{G}$;
 $$\mathcal{G}\left(\sum_{i=1}^{Q}[[u_i]],\sum_{i=1}^{Q}[[v_i]]\right)=\min_{\sigma}\left(\sum_{i=1}^{Q}|u_i-v_{\sigma(i)}|^2\right)^{1/2} $$
 where $\sigma$ runs through all the permutations of
 $\{1,\hdots,Q\}$. A multiple-valued function in the sense of Almgren is a ${\bf Q}_Q(\mathbb{R}^n)$-valued function.\par
 There are a bi-Lipschitz homeomorphism $\xi:{\bf
 Q}_Q(\mathbb{R}^n)\rightarrow Q^*\subset \mathbb{R}^{PQ}$, where $P$
 is a positive integer, and a Lipschitz retraction
 $\rho:\mathbb{R}^{PQ}\rightarrow Q^*$ with $\rho|_{Q^*}=$Id$_{Q^*}$ (see Theorems 1.2(3) and 1.3(1) in \cite{Frederick J. Almgren}).

 \subsection{Affine approximations} The set of all affine maps
 $\mathbb{R}^m\rightarrow \mathbb{R}^n$ is denoted by $A(m,n)$. If
 $v=(v_1,\hdots,v_n)\in A(m,n)$, we put
 $$|v|=\left(\sum_{i=1}^{n}\sum_{j=1}^m \left(\frac{\partial v_i}{\partial
 x_j}\right)^2\right)^{1/2}\in \mathbb{R}.$$
 A function $v:\mathbb{R}^m\rightarrow {\bf Q}_Q(\mathbb{R}^n)$ is
 called affine if there are $v_1,\hdots, v_Q\in A(m,n)$ such that
 $v=\sum_{i=1}^Q [[v_i]]$. Then we set $|v|=\left(\sum_{i=1}^Q
 |v_i|^2\right)^{1/2}$.\par
 If $\Omega\subset \mathbb{R}^m$ is open and $a\in \Omega$, a function
 $u:\Omega\rightarrow {\bf Q}_Q(\mathbb{R}^n)$ is said to be
 approximately affinely approximable at $a$ if there is an affine
 function $v:\mathbb{R}^m\rightarrow {\bf Q}_Q(\mathbb{R}^n)$ such
 that
 $$\text{ap}\lim_{x\rightarrow a}
 \frac{\mathcal{G}(u(x),v(x))}{|x-a|}=0.$$
Such a function $v$ is uniquely determined and denoted by ap$
Au(a)$. This is the case if $\xi\circ u$ is approximately
differentiable at $a$, and then
$$\left(\text{Lip}(\xi)\right)^{-1}|\text{ap}D(\xi\circ u)(a)|\leq
|\text{ap}Au(a)|\leq |\text{ap} D(\xi\circ u)(a)|$$ (see Theorem
1.4(3) in \cite{Frederick J. Almgren}).

\subsection{The Sobolev space $\mathcal{Y}_2\left(U,{\bf Q}_Q(\mathbb{R}^n)\right)$}
Suppose $U$ is an open ball in $\mathbb{R}^m$ or the entire space
$\mathbb{R}^m$.

\begin{enumerate}
\item We denote by $W^{1,2}(U,\mathbb{R}^n)$ the Sobolev space of
$\mathbb{R}^n$-valued functions on $U$ which together with their
first order distributional partial derivatives are $\mathcal{L}^m$
square summable over $U$.

\item A function $u\in W^{1,2}\left(U,\mathbb{R}^n\right)$ is said
to be strictly defined if $u(x)=y$ whenever $x\in U, y\in
\mathbb{R}^n$ and
$$\lim_{r\rightarrow 0} r^{-m}\int_{B^m(x,r)}|u(z)-y|\; d \mathcal{L}^m z=0 .$$
Any $v\in W^{1,2}(U,\mathbb{R}^n)$ agrees $\mathcal{L}^m$ almost
everywhere on $U$ with a strictly defined $u\in
W^{1,2}(U,\mathbb{R}^n)$. If $u\in
W^{1,2}\left(U,\mathbb{R}^n\right)$ is strictly defined, then
$$\lim_{r\rightarrow 0} r^{-m}\int_{B^m(x,r)}|u(z)-u(x)|\; d
\mathcal{L}^m z=0$$ for $\mathcal{H}^{m-1}$ almost all $x\in U$ (see
Section 4.8 in \cite{EvansGariepy}).

\item The space $\mathcal{Y}_2\left(U,{\bf
Q}_Q(\mathbb{R}^n)\right)$ consists of those functions
$u:U\rightarrow {\bf Q}_Q(\mathbb{R}^n)$ for which $\xi\circ u\in
W^{1,2}\left(U,\mathbb{R}^{PQ}\right)$. We say that $u$ is
strictly defined if $\xi\circ u$ is strictly defined. If $u,v\in
\mathcal{Y}_2\left(U,{\bf Q}_Q(\mathbb{R}^n)\right)$, we write
$u=v$ if $u(x)=v(x)$ for $\mathcal{L}^m$ almost all $x\in U$.

\item Suppose $u\in \mathcal{Y}_2\left(U,{\bf
Q}_Q(\mathbb{R}^n)\right)$. Then $u$ is approximately affinely
approximable $\mathcal{L}^m$ almost everywhere on $U$ since
$\xi\circ u$ is (see Section 6.1.3 in \cite{EvansGariepy}). The
Dirichlet integral of $u$ over $U$ is defined by
$$\text{Dir}(u;U)=\int_U |\text{ap}Au(x)|^2 \; d\mathcal{L}^m x.$$
\end{enumerate}

\subsection{The Sobolev space $\partial\mathcal{Y}_2\left(\partial U,{\bf Q}_Q(\mathbb{R}^n)\right)$}
Suppose $U\subset \mathbb{R}^m$ is an open ball.
\begin{enumerate}
\item The space $\partial\mathcal{Y}_2\left(\partial U,{\bf
Q}_Q(\mathbb{R}^n)\right)$ will be the set of those functions
$v:\partial U\rightarrow {\bf Q}_Q(\mathbb{R}^n)$ for which there
is a strictly defined $u\in \mathcal{Y}_2\left(\mathbb{R}^m,{\bf
Q}_Q(\mathbb{R}^n)\right)$ such that $u(x)=v(x)$ for
$\mathcal{H}^{m-1}$ almost all $x\in \partial U$ (cf. Theorem
A.1.2(7) and Section 2.1(2) in \cite{Frederick J. Almgren}). If
$v,w\in
\partial \mathcal{Y}_2\left(\partial U,{\bf
Q}_Q(\mathbb{R}^n)\right)$, we write $v=w$ if $v(x)=w(x)$ for
$\mathcal{H}^{m-1}$ almost all $x\in\partial U$.

\item If $v\in
\partial\mathcal{Y}_2\left(\partial U,{\bf
Q}_Q(\mathbb{R}^n)\right)$ we say that $u$ has boundary values $v$
if there is $w\in \mathcal{Y}_2\left(\mathbb{R}^m,{\bf
Q}_Q(\mathbb{R}^n)\right)$ which is strictly defined such that
$w|_U=u|_U$ and $w|_{\partial U}=v$. We then write $u|_{\partial
U}=v$. If $u$ is strictly defined, $v$ agrees with $u$
$\mathcal{H}^{m-1}$ almost everywhere on $\partial U$.

\item One says that $u:U\rightarrow {\bf Q}_Q(\mathbb{R}^n)$ is
Dirichlet minimizing if and only if $u\in
\mathcal{Y}_2\left(U,{\bf Q}_Q(\mathbb{R}^n)\right)$ and, assuming
$u$ has boundary values $v\in \partial \mathcal{Y}_2\left(\partial
U,{\bf Q}_Q(\mathbb{R}^n)\right)$, one has
$$\text{Dir}\left(u;U\right)=\inf \left\{\text{Dir}\left(w;U\right) :
w\in \mathcal{Y}_2\left(U,{\bf Q}_Q(\mathbb{R}^n)\right)\text{ has
boundary values }v \right\}.$$ If $v\in \partial
\mathcal{Y}_2\left(\partial U,{\bf Q}_Q(\mathbb{R}^n)\right)$ then
there exists a Dirichlet minimizing $u\in
\mathcal{Y}_2\left(U,{\bf Q}_Q(\mathbb{R}^n)\right)$ that has
boundary values $v$ (see Theorem 2.2(2) in \cite{Frederick J.
Almgren}).

\item For each $v\in\partial\mathcal{Y}_2(\partial U,{\bf
Q}_Q(\mathbb{R}^n))$, the Dirichlet integral of $v$ over $\partial
U$ is defined by
$$\text{dir}(v;\partial U)=\int_{\partial U} |\text{ap}Av(x)|^2 \;d\mathcal{H}^{m-1}x$$
where the affine approximation $\text{ap}Av$ is with respect to
$\partial U$.
\end{enumerate}

\section{Dirichlet quasiminimizer}\label{quasireg}
Throughout this section $K$ will be a real number larger than 1.

\begin{definition}
A strictly defined function $u\in\mathcal{Y}_2(U^{m}(0,1),{\bf
Q}_{Q}(\mathbb{R}^n))$ is a Dirichlet  $K$-quasiminimizer if for
every ball $U\subset U^{m}(0,1)$,
$$\text{Dir}(u;U)\leq K \;\text{Dir}(v;U)$$
 where $v$ is a Dirichlet
minimizing multiple-valued function having boundary values
$u|_{\partial U}\in \partial\mathcal{Y}_2\left(\partial U,{\bf
Q}_Q(\mathbb{R}^n)\right)$.
\end{definition}

\subsection{Regularity}

For the convenience of the reader, we shall recall here a few
known results that we are going to need later. We begin with a
"Dirichlet growth" theorem guaranteeing Hölder continuity:

\begin{thmmorrey}
Suppose $u\in W^{1,p}(B^m(x_0,R))$, $0<\sigma<1$ and
$$\int_{B(z,r)}|\nabla u(x)|^p\;dx\leq
C\;(r/\delta)^{m-p+p\sigma},\quad 0\leq r\leq \delta=R-|z-x_0|$$ for
every $z\in B(x_0,R)$. Then $$u\in C^{0,\sigma}(B^m(x_0,r))$$ for
each $r<R$.
\end{thmmorrey}

We recall Almgren's estimates for 2-dimensional and higher
dimensional domains:

\begin{Almgren2}\label{al2}
Suppose $v\in \partial\mathcal{Y}_2\left(\partial U^2(x,r),{\bf
Q}_{Q}(\mathbb{R}^n)\right)$. Then there exists a multiple-valued
function $u\in \mathcal{Y}_2\left(U^2(x,r),{\bf
Q}_{Q}(\mathbb{R}^n)\right)$ such that
\begin{enumerate}
\item $u$ has boundary values $v$, \item Dir$\left(u;
U^2(x,r)\right)\leq Qr\;\text{dir}\left(v;\partial U^2(x,r)\right)$.
\end{enumerate}
\end{Almgren2}

\begin{Almgrenm}
Let $m\geq 3$ be an integer. If $f\in \mathcal{Y}_2(U^m(x,r),{\bf
Q}_{Q}(\mathbb{R}^{n}))$ is strictly defined and Dirichlet
minimizing then
$$\text{Dir}(f;U^m(x,r))\leq\left(\frac{1-\varepsilon_Q}{m-2}\right)r\;\text{dir}(f;\partial
U^{m}(x,r))$$ where $0<\varepsilon_Q<1$ is a constant.
\end{Almgrenm}

\begin{thm}\label{Kregulm2}
Suppose that $u\in \mathcal{Y}_2\left(U^2(0,1),{\bf
Q}_{Q}(\mathbb{R}^n)\right)$ is a strictly defined Dirichlet
 $K$-quasiminimizer such that
$\text{Dir}\left(u;U^2(0,1)\right)>0$. Then the following assertions
are checked:
\begin{enumerate}
\item Suppose $z\in U^{2}(0,1)$. Then for $\mathcal{L}^{1}$ almost
all $0<r< 1-|z|$,
$$\text{Dir}(u;U^2(z,r))\leq KQr\;
\frac{d}{ds}\;\text{Dir}(u;U^2(z,s))|_{s=r}.$$ \item For each $z\in
U^2(0,1)$, $0<r<1-|z|$, and $0<s\leq 1$,
$$\text{Dir}(u;U^{2}(z,sr))\leq
s^{\frac{1}{KQ}}\;\text{Dir}(u;U^{2}(z,r)).$$ \item
$f|_{B^2(0,\delta)}$ is Hölder continuous with exponent
$\frac{1}{2KQ}$ for all $0<\delta<1$.
\end{enumerate}
\end{thm}
\begin{proof}
Let $v\in \mathcal{Y}_2(U^2(z,r),{\bf Q}_{Q}(\mathbb{R}^n))$ be a
Dirichlet minimizing multiple-valued function with boundary values
$u|_{\partial U^2(z,r)}$. It follows from Section 2.7(2) in
\cite{Frederick J. Almgren} that
$$\text{Dir}(u;U^2(z,r))\leq K\;\text{Dir}(v;U^{2}(z,r))\leq KQr\;\text{dir}(u;\partial
U^2(z,r)).$$ Using the fact that the function
$\phi(r)=$Dir$(u;U^2(z,r))$ is absolutely continuous with
$$\phi'(r)= \int_{\partial U^2(z,r)}|\text{ap}Au(x)|^2\;d\mathcal{H}^{1}x \geq \text{dir}(u;\partial U^2(z,r))$$
for $\mathcal{L}^1$ almost all $0<r<1-|z|$, we obtain the first
assertion. One derives (2) from (1) by integration.  In order to
verify (3), one notes that conclusion (2) implies for each $z\in
U^2(0,1),$ $0<r<1-|z|$, and $0<s\leq 1$,
\begin{eqnarray*}
\text{Dir}(\xi\circ u;U^2(z,sr))& \leq &
\text{Lip}(\xi)^2\;\text{Dir}(u;U^2(z,sr))\\
 & \leq & \text{Lip}(\xi)^2\;
 s^{\frac{1}{KQ}}\;\text{Dir}(u;U^2(z,r))\\
  & \leq & \text{Lip}(\xi)^2\; s^{\frac{1}{KQ}}\;\text{Dir}(\xi\circ
 u;U^2(z,r))\\
 & \leq & \text{Lip}(\xi)^2\;\text{Dir}(\xi\circ
 u;U^2(0,1))\;s^{\frac{1}{KQ}}.
\end{eqnarray*}
Let us fix $0<\delta<1$. Theorem 3.5.2 in \cite{Morrey} applied to
the function $\xi\circ u$ implies that $$\xi\circ u\in
C^{0,\frac{1}{2KQ}}(B^2(0,\delta))$$ and the third assertion is
checked since $\xi$ is bi-Lipschitz.
\end{proof}

The regularity result is weaker for higher dimensional domains since
we need an extra assumption (\ref{small}) on the size of $K$. This
assumption on $K$ is necessary since M. Giaquinta presents in
\cite{Giaquintabook} a single-valued quasiminimizer for the
Dirichlet integral
 which is singular on a dense set.

\begin{thm}\label{regulm3}
Let $m\geq 3$ be an integer. Suppose that $u\in
\mathcal{Y}_2(U^m(0,1),{\bf Q}_{Q}(\mathbb{R}^n))$ is a strictly
defined Dirichlet  $K$-quasiminimizer such that
$\text{Dir}(u;U^m(0,1))>0$. If
\begin{equation}\label{small}
1\leq K<\frac{1}{(1-\varepsilon_Q)}
\end{equation} then there exists a constant
$0<\sigma=\sigma\left(K,\varepsilon_Q,m\right)<1$ such that the
following assertions are checked:

\begin{enumerate}
\item Suppose $z\in U^{m}(0,1)$. Then for $\mathcal{L}^{1}$ almost
all $0<r< 1-|z|$,
$$\text{Dir}(u;U^m(z,r))\leq \frac{1}{(m-2+2\sigma)}\; r\;
\frac{d}{ds}\;\text{Dir}(u;U^m(z,s))|_{s=r}.$$
\item For each $z\in
U^m(0,1)$, $0<r<1-|z|$, and $0<s\leq 1$,
$$\text{Dir}(u;U^{m}(z,sr))\leq
s^{m-2+2\sigma}\;\text{Dir}(u;U^{m}(z,r)).$$
\item
$f|_{B^m(0,\delta)}$ is Hölder continuous with exponent $\sigma$ for
all $0<\delta<1$.
\end{enumerate}
\end{thm}

\begin{proof}
We fix  $0<\sigma<1$ requiring that
$$\frac{m-2}{K(1-\varepsilon_Q)}>m-2+2\sigma.$$
Let $v\in \mathcal{Y}_2(U^m(z,r),{\bf Q}_{Q}(\mathbb{R}^n))$ be a
Dirichlet minimizing multiple-valued function with boundary values
$u|_{\partial U^m(z,r)}$. It follows from Theorem 2.12 in
\cite{Frederick J. Almgren} that
\begin{eqnarray*}
\text{Dir}(u;U^m(z,r)) & \leq & K\;\text{Dir}(v;U^{m}(z,r))\\
&  \leq &
K\left(\frac{1-\varepsilon_Q}{m-2}\right)r\;\text{dir}(v;\partial
U^{m}(z,r))\\
 & \leq & \frac{1}{(m-2+2\sigma)}\; r\;\text{dir}(u;\partial
U^{m}(z,r)).
\end{eqnarray*}
Using the fact that the function $\phi(r)=$Dir$(u;U^m(z,r))$ is
absolutely continuous with
$$\phi'(r)=\int_{\partial U^m(z,r)}|\text{ap}Au(x)|^2\;d\mathcal{H}^{m-1}x \geq \text{dir}(u;\partial U^m(z,r))$$
for $\mathcal{L}^1$ almost all $0<r<1-|z|$, we obtain the first
claim. One derives (2) from (1) by integration. In order to
conclude, one notes that the second claim implies for each $z\in
U^m(0,1)$, $0<r<1-|z|$, and $0<s \leq 1$,
\begin{eqnarray*}
\text{Dir}(\xi\circ u;U^m(z,sr))& \leq &
\text{Lip}(\xi)^2\;\text{Dir}(u;U^m(z,sr))\\
 & \leq & \text{Lip}(\xi)^2\;
s^{m-2+2\sigma}\;\text{Dir}(u;U^m(z,r))\\
  & \leq & \text{Lip}(\xi)^2\; s^{m-2+2\sigma}\;\text{Dir}(\xi\circ
 u;U^m(z,r))\\
 & \leq &  \text{Lip}(\xi)^2\;\text{Dir}(\xi\circ
 u;U^m(0,1))\;s^{m-2+2\sigma}.
\end{eqnarray*}
Let us fix $0<\delta<1$. Theorem 3.5.2 in \cite{Morrey} applied to
the function $\xi\circ u$ implies that $$\xi\circ u\in C^{0,\sigma
}(B^m(0,\delta))$$ and the third assertion is checked since $\xi$ is
bi-Lipschitz.
\end{proof}

\subsection{Branch set of quasiminimizer}\label{quasiex}

\begin{definition}
Assume the multiple-valued function $u:U^m(0,1)\rightarrow {\bf
Q}_{Q}(\mathbb{R}^{n})$ is continuous. Define the function
$\sigma:U^m(0,1)\rightarrow\{1,\hdots,Q\}$ as
$$\sigma(x)=\text{card }(\text{spt}\;(u(x)))$$
for all $x\in U^m(0,1)$. Then, the closed set
$$\Sigma=\{x\in U^m(0,1)\;|\;\sigma\text{ is discontinuous at }x\}$$
is called the {\it branch set} of the continuous multiple-valued
function $u$.
\end{definition}

Almgren proved that the Hausdorff dimension of the branch set of a
strictly defined and Dirichlet minimizing $u\in
\mathcal{Y}_2(U^m(0,1),{\bf Q}_{Q}(\mathbb{R}^n))$ does not exceed
$m-2$ (see Theorem 2.14 in \cite{Frederick J. Almgren}). Here we
will construct a one-dimensional Dirichlet quasiminimizer with a
fractal branch set. This construction can be adapted in order to
obtain a one-dimensional Dirichlet quasiminimizer with a branch
set of positive $1$-dimensional Lebesgue measure.\par
The
following result characterizes Dirichlet minimizing
multiple-valued functions in codimension one.

\begin{Dir1}\label{codimension1}
Suppose $v\in \partial\mathcal{Y}_2(\partial U^m(0,1),{\bf
Q}_{Q}(\mathbb{R}))$ and $v_{1},\hdots,v_Q:\partial
U^m(0,1)\rightarrow \mathbb{R}$ are defined by requiring for each
$x\in \partial U^m(0,1)$ that $v(x)=\sum_{i=1}^{Q}[[ v_i(x)]]$ and
$v_1(x)\leq v_2(x)\leq \hdots\leq v_Q(x)$. Then $v_i\in
\partial\mathcal{Y}_2(\partial U^m(0,1),\mathbb{R})$ for each
$i=1,\hdots,Q$ and $u\in\mathcal{Y}_2(U^m(0,1),{\bf
Q}_{Q}(\mathbb{R}))$ is strictly defined and Dirichlet minimizing
with boundary values $v$ if and only if $u=\sum_{i=1}^Q[[ u_i]] $
corresponding to $u_1, u_2,\hdots, u_Q\in
\mathcal{Y}_2(U^m(0,1),\mathbb{R})$ which are the unique harmonic
functions having boundary values $v_1,v_2,\hdots, v_Q$
respectively.
\end{Dir1}

This last theorem implies the following characterization of
Dirichlet  quasiminimizers in dimension and codimension one.

\begin{thm}\label{Kmindim1}
A multiple-valued function $u\in\mathcal{Y}_2((-1,1),{\bf
Q}_{Q}(\mathbb{R}))$ is a Dirichlet  K-quasiminimizer if and only
if for every interval $(a,b) \subset (-1,1)$,
$$\text{Dir}(u,(a,b))\leq K\; \frac{\mathcal{G}^2(u(b),u(a))}{b-a}.$$
\end{thm}

\begin{proof}
Let $(a,b)\subset (-1,1)$. We define the functions
$v_{1},\hdots,v_Q:\partial (a,b)\rightarrow \mathbb{R}$ by requiring
for each $x\in \partial (a,b)$ that $u(x)=\sum_{i=1}^{Q}[[ v_i(x)]]$
and $v_1(x)\leq v_2(x)\leq \hdots\leq v_Q(x)$. By Theorem 2.16 in
\cite{Frederick J. Almgren}, the Dirichlet minimizing multiple-valued
function coinciding with $u$ on $\partial (a,b)$ is given by
$$
v(\cdot)=\sum_{i=1}^{Q}\left[\left[
\frac{v_{i}(b)-v_{i}(a)}{b-a}(\cdot -a)+v_{i}(a)\right]\right]$$ and
\begin{multline*}
\text{Dir}(v;[a,b]) =  \int_{a}^{b}|\text{ap}Av(x)|^2\;dx
 = \frac{(b-a)}{(b-a)^{2}}\;\sum_{i=1}^{Q}|
  v_{i}(b)-v_{i}(a)|^{2}\\
   = \frac{\mathcal{G}^2(u(b),u(a))}{b-a}.
\end{multline*}
\end{proof}

We will now produce a Dirichlet  quasiminimizer with a fractal
branch set. We begin with some definitions:

\begin{enumerate}
\item A {\it diamond} above an interval $I=[a,b]\subset\mathbb{R}$
is a multiple-valued function $u:[a,b]\rightarrow {\bf
Q}_2(\mathbb{R})$ whose the graph is a parallelogram with the
following vertices $[a,h], [(a+b)/2,h], [(a+b)/2, h+(b-a)/2],
[b,h+(b-a)/2]$ where $h\in\mathbb{R}$.
\item  A {\it pluri-diamond}
is a multiple-valued function $u:[0,1]\rightarrow {\bf
Q}_2(\mathbb{R})$ which admits a partition of $[0,1]$ in intervals
$\{I_j\}$ such that
\begin{enumerate}
\item[-] $u$ is a diamond above $I_j$ or a map of the form
$x\rightarrow 2[[ x+p_j]]$ above $I_j$ where $p_j\in \mathbb{R}$,
\item[-] $u$ is continuous on $[0,1]$.
\end{enumerate}
\end{enumerate}

Figure 1 gives examples of pluri-diamonds. It is clear that a
pluri-diamond $u$ is Lipschitz with Lip$(u)\leq \sqrt{2}$. We record
two lemmas for later use:

\begin{lem}\label{pluri-diamond}
A pluri-diamond $u$ is a Dirichlet  4-quasiminimizer.
\end{lem}
\begin{proof}
Let $u_{\min}, u_{\max}:[0,1]\rightarrow\mathbb{R}$ be such that
$u=[[ u_{\min}]]+[[ u_{\max}]]$ and $u_{\min}(x)\leq u_{\max}(x)$
for all $x\in [0,1]$. For every interval $[a,b]\subset [0,1]$, we
have
\begin{multline*}
\mathcal{G}^2(u(b),u(a)) \geq \frac{1}{2}\left(
|u_{\min}(b)-u_{\min}(a)|+|u_{\max}(b)-u_{\max}(a)|\right)^2
\\=\frac{1}{2} \left(\int_{a}^{b} u_{\min}'(x)+ u_{\max}'(x) \;dx\right)^2
\geq \frac{(b-a)^2}{2}
\end{multline*}

and
\begin{multline*}
\frac{(b-a)\int_{a}^{b}
|\text{ap}Au(x)|^{2}\;dx}{\mathcal{G}^2(u(b),u(a))} =
\frac{(b-a)\int_{a}^{b} u_{\min}'^{2}(x)+u_{\max}'^{2}(x)
\;dx}{\mathcal{G}^2(u(b),u(a))}\\ \leq
\frac{2\;(b-a)^{2}}{\mathcal{G}^2(u(b),u(a))}
          \leq  4
\end{multline*}
hence $u$ is a Dirichlet  4-quasiminimizer by Theorem
\ref{Kmindim1}.
\end{proof}

\begin{lemwillem} Let
$\Omega$ be an open subset of $\mathbb{R}^m$, $1<p<\infty$. If
$\{u_i\}$ is a bounded sequence in $W^{1,p}(\Omega)$ such that
$u_i\rightharpoonup u$ in $L^p(\Omega)$ then $u\in W^{1,p}(\Omega)$
and $$\|\nabla u\|_p\leq \liminf_{i\rightarrow\infty}\|\nabla
u_i\|_p.$$
\end{lemwillem}

We will now create a particular sequence of pluri-diamonds
$\{u^i\}$. We define $u^1:\mathbb{R}\rightarrow{\bf
Q}_2(\mathbb{R})$ by requiring that
\begin{itemize}
\item[-] $u^1$ is a pluri-diamond, \item[-] $u^1(1/3)=0$, \item[-]
$u^1$ is a diamond only above $[1/3,2/3]$.
\end{itemize}
We define $u^2:\mathbb{R}\rightarrow{\bf Q}_2(\mathbb{R})$ by
requiring that
\begin{itemize}
\item[-] $u^2$ is a pluri-diamond, \item[-]
$u^2(1/3)=0$, \item[-] $u^2$ is a diamond only above the intervals
$[1/3,2/3],[1/9,2/9]$ and $[7/9,8/9]$.
\end{itemize}

\begin{figure}[htbp]\label{cantor}
    \begin{center}
    \includegraphics[height=4cm,width=4cm]{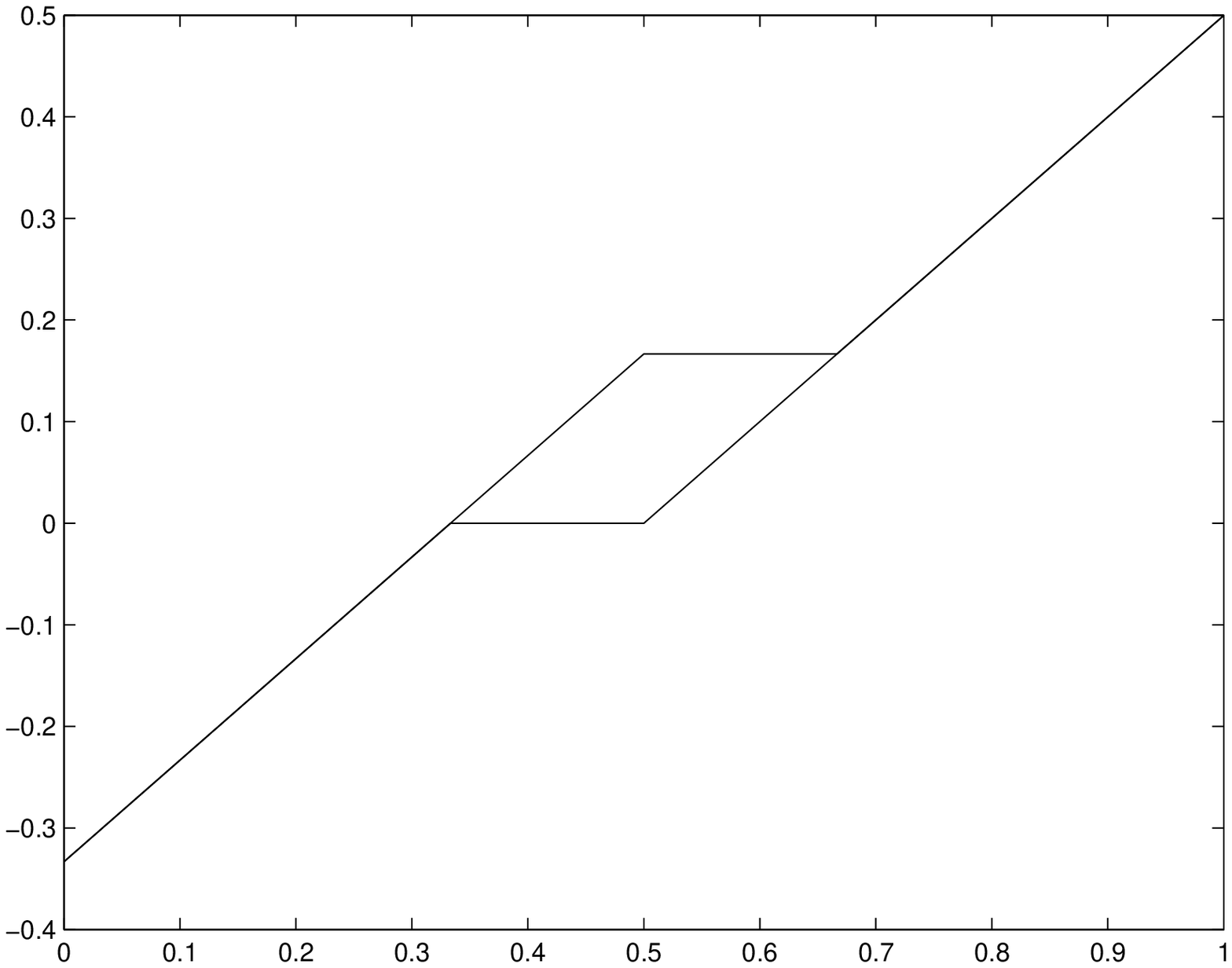}
    \includegraphics[height=4cm,width=4cm]{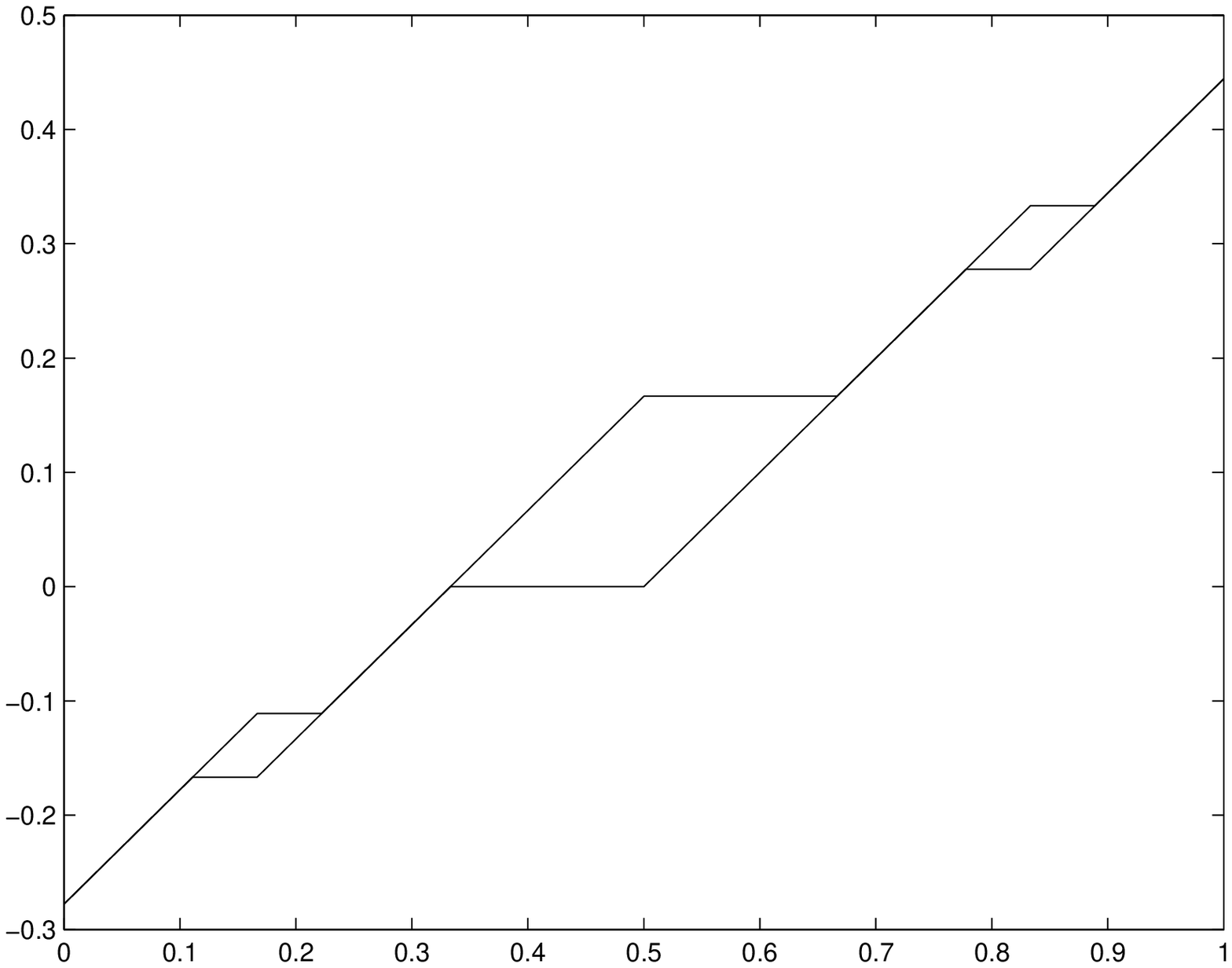}
    \includegraphics[height=4cm,width=4cm]{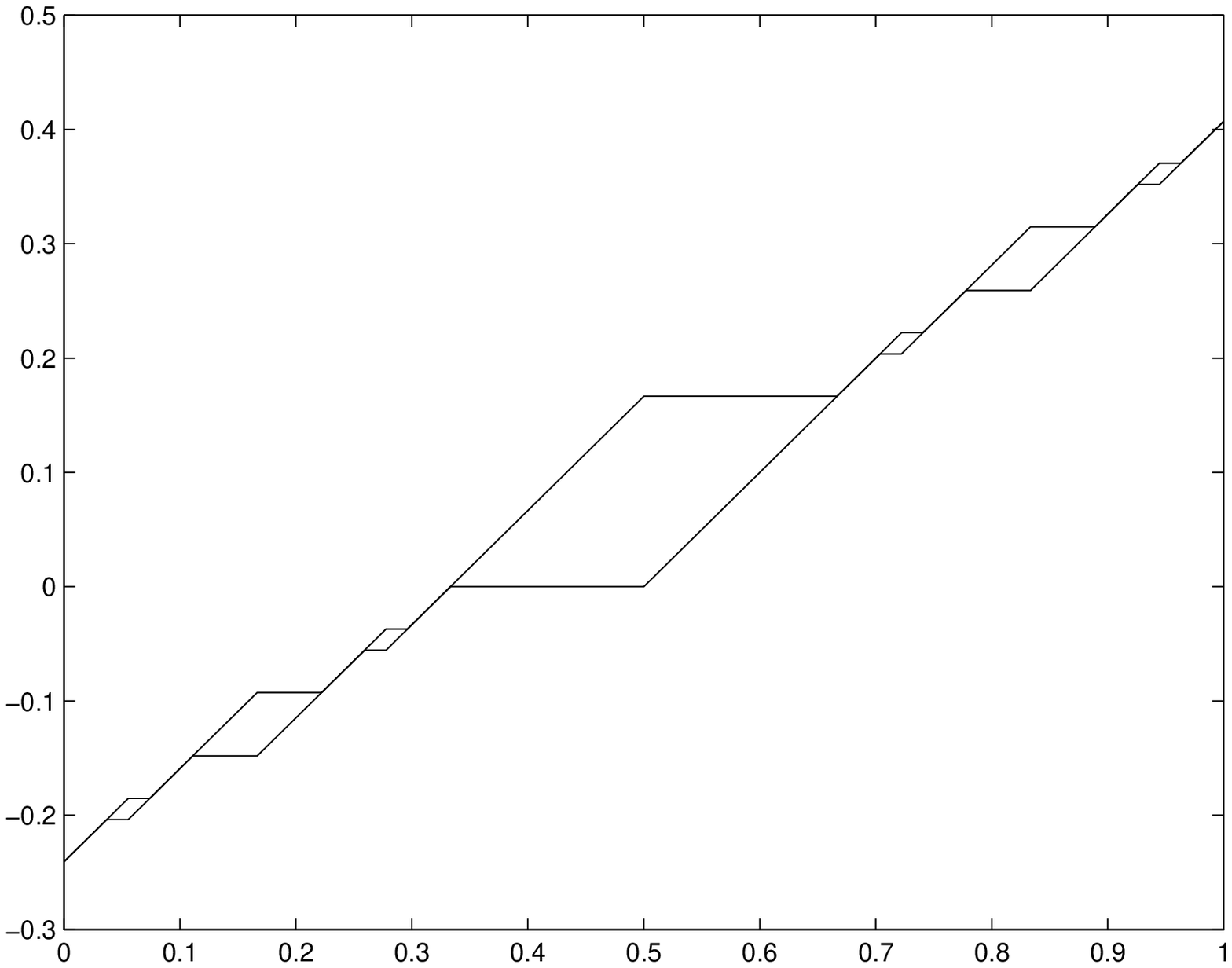}
    \caption{The pluri-diamonds $u^1$, $u^2$ and $u^3$.}
    \end{center}
\end{figure}
The reader will easily imagine how the remaining terms of the
sequence are defined. For each $i\in\mathbb{N}$, we define
$u_{\min}^i, u_{\max}^i:[0,1]\rightarrow\mathbb{R}$ by requiring
that
$$u^i(x)=[[ u_{\min}^i(x)]]+[[
u_{\max}^i(x)]]\;\;\text{
 and  }\;\;u_{\min}^i(x)\leq u_{\max}^i(x)$$ for all $x\in [0,1]$. One
can checks that
$$\text{Lip}(u_{\min}^i)=\text{Lip}(u_{\max}^i)=1$$ for each
$i\in\mathbb{N}$. By Ascoli's Theorem, we can extract two
subsequences still denoted $\{u_{\min}^i\}$ and $\{u_{\max}^i\}$ and
there exists two 1-Lipschitz functions
$u_{\min},u_{\max}:[0,1]\rightarrow\mathbb{R}$ such that
$\{u_{\min}^i\}$ and $\{u_{\max}^i\}$ converge uniformly
 to $u_{\min}$ and $u_{\max}$. We define the Lipschitz
multiple-valued function $u:[0,1]\rightarrow {\bf Q}_2(\mathbb{R})$
by
\begin{equation}\label{exthierry}
u=[[ u_{\min}]]+[[ u_{\max}]].
\end{equation}

\begin{prop}
The multiple-valued function $u$ defined by (\ref{exthierry}) is a
Dirichlet  4-quasiminimizer and the branch set of $u$ is the
Cantor ternary set.
\end{prop}

\begin{proof}
We estimate for any interval $[a,b]\subset [0,1]$,
\begin{eqnarray*}
\int_a^b |\text{ap}Au(x)|^2\;dx & = & \int_a^b
\left(u_{\min}'(x)\right)^2+\left(u_{\max}'(x)\right)^2\;dx\\
 & \leq & \liminf_{i\rightarrow\infty} \int_a^b
\left(u^i_{\min}\;'(x)\right)^2+\left(u^i_{\max}\;'(x)\right)^2\;dx\\
& = & \liminf_{i\rightarrow \infty} \int_a^b
|\text{ap}Au^i(x)|^2\;dx\\
 & \leq & 4\;\liminf_{i\rightarrow\infty}\;
 \frac{\mathcal{G}^2(u^i(b),u^i(a))}{b-a}\\
 & = & 4\;\liminf_{i\rightarrow\infty}\;
 \frac{|u^i_{\min}(b)-u^i_{\min}(a)|^2+|u^i_{\max}(b)-u^i_{\max}(a)|^2}{b-a}\\
 & = & 4\;
 \frac{|u_{\min}(b)-u_{\min}(a)|^2+|u_{\max}(b)-u_{\max}(a)|^2}{b-a}\\
 & = & 4\; \frac{\mathcal{G}^2(u(b),u(a))}{b-a}
\end{eqnarray*}
where the first inequality follows from Lemma 23.7 in
\cite{Willem} and the second inequality is a consequence of Lemma
\ref{pluri-diamond}. By Theorem \ref{Kmindim1}, the
multiple-valued function $u$ is a Dirichlet  4-quasiminimizer. We
denote the Cantor ternary set by
$$T=\cap_{i=1}^{\infty} T_i$$ where
$T_1=[0,1/3]\cup [2/3,1], T_2=[0,1/9]\cup[2/9,1/3]\cup [2/3,7/9]\cup
[8/9,1]$, etc. We can make the following observations:
\begin{enumerate}
\item If $x\in T$ then $u_{\min}^i(x)=u_{\max}^i(x)$ for all
$i\in\mathbb{N}_0$ hence $u_{\min}(x)=u_{\max}(x)$. \item If
$x\notin T$ then there exists $i\in\mathbb{N}_0$ such that $x\notin
T_j$ for each $j\geq i$. By construction, there exists a constant
$\eta>0$ such that $u_{\max}^j(x)-u_{\min}^j(x)=\eta$ for each
$j\geq i$ hence $u_{\min}(x)\neq u_{\max}(x)$.
\end{enumerate}
We deduce from (1) and (2) that
$$\left\{x\in[0,1]\;|\;\text{card}\;(\text{spt}(u(x)))=1\right\}=T$$
 hence the branch set of $u$ is $T$ since it contains no interval.
\end{proof}
It is clear that this construction can be adapted in order to
obtain a quasiminimizer with a fat Cantor branch set.

\subsection{Application}

Let $a\in L^{\infty}(U^m(0,1),\mathbb{R})$ and consider the
following functional
$$J(u;U^m(0,1))=\int_{U^m(0,1)} a(x)\;|\text{ap}Au(x)|^2 \; d\mathcal{L}^m x$$
defined for each $u\in \mathcal{Y}_2\left(U^m(0,1),{\bf
Q}_Q(\mathbb{R}^n)\right)$. One says that $u:U^m(0,1)\rightarrow
{\bf Q}_Q(\mathbb{R}^n)$ is $J-$minimizing if and only if $u\in
\mathcal{Y}_2\left(U^m(0,1),{\bf Q}_Q(\mathbb{R}^n)\right)$ and,
assuming $u$ has boundary values $v\in \partial
\mathcal{Y}_2\left(\partial U^m(0,1),{\bf
Q}_Q(\mathbb{R}^n)\right)$, one has
\begin{multline*}
J\left(u;U^m(0,1)\right)=\inf \{J\left(w;U^m(0,1)\right) : w\in
\mathcal{Y}_2\left(U^m(0,1),{\bf Q}_Q(\mathbb{R}^n)\right)\\
\text{ has boundary values }v \}.
\end{multline*}
\begin{thm}
Let $m\geq 2$ be an integer. Assume $v\in \partial
\mathcal{Y}_2\left(\partial U^m(0,1),{\bf
Q}_Q(\mathbb{R}^n)\right)$ and $a(x)\geq \gamma>0$ for almost all
$x\in U^m(0,1)$. Then there exists a strictly defined
$J-$minimizing $u\in \mathcal{Y}_2\left(U^m(0,1),{\bf
Q}_Q(\mathbb{R}^n)\right)$ that has boundary values $v$ and
\begin{enumerate}
\item if $m=2$, $u|_{B^2(0,\delta)}$ is Hölder continuous for all
$0<\delta<1$; \item if $m\geq 3$ and
$$\frac{\|a\|_{\infty}}{\gamma}\leq \frac{1}{1-\varepsilon_Q}$$
then $u|_{B^m(0,\delta)}$ is Hölder continuous for all
$0<\delta<1$.
\end{enumerate}
\end{thm}
\begin{proof}
The existence proof is a direct adaptation of the proof of Theorem
2.2(2) in \cite{Frederick J. Almgren}. The regularity follows from
the fact that $u$ is a Dirichlet
$\frac{\|a\|_{\infty}}{\gamma}$-quasiminimizer. Consequently it
remains to apply Theorem \ref{Kregulm2} or Theorem  \ref{regulm3}.
\end{proof}

\section{Dirichlet $\omega-$minimizer}\label{Wmini}
Throughout this section $\omega:(0,\infty)\rightarrow\mathbb{R}$
will be a non-negative function.

\begin{definition}
A strictly defined function $u\in\mathcal{Y}_2(U^{m}(0,1),{\bf
Q}_{Q}(\mathbb{R}^n))$ is a Dirichlet $\omega$-minimizer if for
every ball $U^m(x,r)\subset U^{m}(0,1)$,
$$\text{Dir}(u;U^m(x,r))\leq (1+\omega(r)) \;\text{Dir}(v;U^m(x,r))$$
 where $v$ is a Dirichlet
minimizing multiple-valued function having boundary values
$u|_{\partial U^m(x,r)}\in \partial\mathcal{Y}_2\left(\partial
U^m(x,r),{\bf Q}_Q(\mathbb{R}^n)\right)$.
\end{definition}

\subsection{Regularity}

The following regularity result is a direct consequence of Theorem
\ref{Kregulm2} and Theorem \ref{regulm3}.

\begin{thm}\label{regul}
Let $m\geq 2$ be an integer. Suppose that $u\in
\mathcal{Y}_2(U^m(0,1),{\bf Q}_{Q}(\mathbb{R}^n))$ is a strictly
defined $\omega$-minimizer such that $\text{Dir}(u;U^m(0,1))>0$.
If\\
$\lim_{r\rightarrow 0}\omega(r)=0$ then
$$u\in
C_{\text{loc}}^{0,\sigma}(U^m(0,1),{\bf Q}_{Q}(\mathbb{R}^n))$$ for
some $0<\sigma<1$.
\end{thm}

\subsection{Branch set of $\omega-$minimizer}\label{branch_w}
We will construct a one-dimensional Dirichlet $\omega-$minimizer
with a unique branch point at the origin. First of all, we show
that a one-dimensional single-valued $\omega-$minimizer does not
have to be a straight line even if $\lim_{r\to 0} \omega(r)=0$.

\begin{prop}\label{weiex}
$\sin x:(-\pi/4,\pi/4)\to\mathbb{R}$ is a Dirichlet
$\omega-$minimizer with
$$\omega(r)=\frac{r^2}{\sin^2 r}-1\searrow 0\text{ as }r\searrow 0.$$
\end{prop}
\begin{proof}
It is easy to check that
\begin{equation*}
\begin{split}
\text{Dir}(\sin(\cdot);U^1(x,r))&=\int_{x-r}^{x+r} (\cos s)^2 ds\\
&=\frac{\sin 2(x+r)-\sin 2(x-r)}{4}+r\\
&=\frac{\cos 2x\sin 2r}{2}+r
\end{split}
\end{equation*}
The Dirichlet minimizer $v$ with boundary values
$v(x-r)=\sin(x-r)$ and $v(x+r)=\sin(x+r)$ is the straight line
with slope
$$\frac{\sin(x+r)-\sin(x-r)}{2r},$$
hence
\begin{equation*}
\begin{split}
\text{Dir}(v;U^1(x,r))&=\left[\frac{\sin(x+r)-\sin(x-r)}{2r}\right]^2 2r\\
&=\frac{2\cos^2 x\sin^2 r}{r}.
\end{split}
\end{equation*}
Now it suffices to show that
\begin{equation}\label{inequality_w}
\frac{\cos 2x\sin 2r}{2}+r\le (1+\omega(r))\frac{2\cos^2 x\sin^2 r}{r},
\end{equation}
for any $U^1(x,r)\subset (-\pi/4,\pi/4)$.
Define the function
$$W(x,r)=\left[\frac{\cos 2x\sin 2r}{2}+r\right]/\left[\frac{2\cos^2 x\sin^2 r}{r}\right].$$
Denote
$$W(x,r)=\frac{1}{4\sin^2 r}H(x,r),$$
where
$$H(x,r)=\frac{r\cos(2x)\sin(2r)+2r^2}{\cos^2 x}.$$
We have
\begin{equation*}
\begin{split}
\frac{\partial H}{\partial x}=\frac{2r\sin x}{\cos^3
x}(2r-\sin(2r)).
\end{split}
\end{equation*}
For $r\in(0,1]$,
$$\frac{\partial H}{\partial x}\le 0, \text{ for}\;x\in(-\pi/4,0],$$
$$\frac{\partial H}{\partial x}\ge 0, \text{ for}\;x\in[0,\pi/4).$$
Since $W(x,r)$ is an even function of $x$, we conclude
$$W(x,r)\le W(\pi/4,r)=\frac{r^2}{\sin^2 r},\;\forall x\in(-\pi/4,\pi/4)$$
which proves (\ref{inequality_w}).
\end{proof}

We are ready to construct a one-dimensional ${\bf
Q}_2(\mathbb{R})-$valued $\omega-$minimizer with a branch point
and $\lim_{r\to 0}\omega(r)=0$.

\begin{prop}\label{weiex2}
The multiple-valued function
$$u(x)=[[x]]+[[\sin x]]:(-\pi/4,\pi/4)\to{\bf Q}_2(\mathbb{R})$$
is a $\omega-$minimizer with
$$\omega(r)=\frac{r^2}{\sin^2 r}-1\searrow 0\text{ as }r\searrow 0.$$
The origin is a branch point of $u$.
\end{prop}

\begin{proof}
Proposition \ref{weiex2} is a direct consequence of Proposition
\ref{weiex} since 
the Dirichlet minimizer $v$ with boundary values
$u(x-r)=[[x-r]]+[[\sin(x-r)]],u(x+r)=[[x+r]]+[[\sin(x+r)]]$ is two
straight lines connecting $x-r$ with $x+r$ and $\sin(x-r)$ with
$\sin(x+r)$.
\end{proof}

Let us mention that the authors ignore the maximum Hausdorff
dimension of branch sets of $\omega$-minimizers.

\section{Dirichlet almost minimizer}\label{almost}

Throughout this section $c\geq 0$ and $0<\alpha<1$ are real numbers.

\begin{definition}
A strictly defined function  $u\in\mathcal{Y}_2(U^m(0,1),{\bf
Q}_Q(\mathbb{R}^n))$ is a Dirichlet $(c,\alpha)$-almost minimizer if
for every ball $U^m(x,r)\subset U^m(0,1)$,
$$\text{Dir}(u;U^m(x,r))\le\text{Dir}(v;U^m(x,r))+cr^{m-2+\alpha}$$
where $v$ is a Dirichlet minimizing multiple-valued function having
boundary values $u|_{\partial U^m(x,r)}\in
\partial \mathcal{Y}_2(\partial U^m(x,r),{\bf Q}_Q(\mathbb{R}^n))$ .
\end{definition}


\subsection{Regularity}

For the convenience of the reader, we shall recall a few results
that we are going to need later.

\begin{Almgren3}
Corresponding to numbers $0<s_0<\infty$, $1<K<\infty$, and (not
necessarily distinct points) $q_1,\hdots,q_Q\in\mathbb{R}^n$ we can
find $J\in\{1,\hdots,Q\},k_1,\hdots,k_J\in\{1,\hdots,Q\}$, distinct
points $p_1,\hdots,p_J\in\{q_1,\hdots,q_Q\}$, and $s_0\le r\le Cs_0$
such that
\begin{enumerate}
\item $|p_i-p_j|>2Kr$ for each $1\le i<j\le J$,
\item $\mathcal{G}\left(\sum_{i=1}^Q [[q_i]],\sum_{i=1}^J k_i[[p_i]]\right)\le C(Q)s_0/(Q-1)^{1/2},$
\item $z\in{\bf Q}_Q(\mathbb{R}^n)$ with
$\mathcal{G}\left(z,\sum_{i=1}^Q[[q_i]]\right)\le s_0$ implies
$\mathcal{G}\left(z,\sum_{i=1}^J k_i[[p_i]]\right) \le r$,
\item in case $J=1$,  diam$\left(\text{spt}\left(\sum_{i=1}^Q [[q_i]]\right)\right)\le
C(Q)s_0/(Q-1)$; here
$$C(Q)=1+\left[(2K)(Q-1)^2\right]^1+\left[(2K)(Q-1)^2\right]^2+\hdots+\left[(2K)(Q-1)^2\right]^{Q-1}.$$
\end{enumerate}
\end{Almgren3}

\begin{Almgren4}
Corresponding to
\begin{enumerate}
\item $J\in\{1,2,\hdots,Q\}$,
\item $k_1,k_2,\hdots,k_J\in\{1,2,\hdots,Q\}$ with
$k_1+k_2+\hdots+k_J=Q$,
\item distinct points
$p_1,p_2,\hdots,p_J\in\mathbb{R}^n$,
\item $0<s_1<s_2=2^{-1}\inf\left\{|p_i-p_j|:1\le i<j\le J\right\}$,
\end{enumerate}
we set
\begin{multline*}
\mathbb{P}={\bf Q}_Q(\mathbb{R}^n)\cap
\{\sum_{i=1}^Q[[q_i]]:q_1,\hdots,q_Q\in\mathbb{R}^n\;\mbox {with}\\
\text{card}\left\{i:q_i\in B^n(p_j,s_1)\right\}=k_j\;\mbox{for
each}\;j=1,\hdots,J\}.
\end{multline*}
Then there exists a map $\Phi:{\bf Q}_Q(\mathbb{R}^n)\to\mathbb{P}$
such that
\begin{enumerate}
\item $\Phi(q)=q$ whenever $q\in {\bf Q}_Q(\mathbb{R}^n)$ with
$\mathcal{G}\left(q,\sum_{i=1}^J k_i[[p_i]]\right)\le s_1$,
\item
$\Phi(q)=\sum_{j=1}^J k_j[[p_j]]$ whenever $q\in{\bf
Q}_Q(\mathbb{R}^n)$ with $\mathcal{G}\left(q,\sum_{i=1}^J
k_i[[p_i]]\right)\ge s_2$,
\item $\mathcal{G}\left(q,\Phi(q)\right)\le \mathcal{G}\left(q,\sum_{i=1}^J
k_i[[p_i]]\right)$ for each $q\in{\bf Q}_Q(\mathbb{R}^n)$, \item Lip
$\Phi\le 1+Q^{1/2}s_1/(s_2-s_1).$
\end{enumerate}
\end{Almgren4}

\begin{Almgrenm}\label{estimate}
 For $z\in{\bf Q}_Q(\mathbb{R}^n)$ with $\mathcal{G}(z,q_0)>r$,
 \begin{enumerate}
\item $\mathcal{G}(z,q)> s_0$.
\item $\mathcal{G}(z,q_0)\le\mathcal{G}(z,q)+\mathcal{G}(q,q_0)\le\mathcal{G}(z,q)+\DF{C(Q)s_0}{(Q-1)^{1/2}}\le[1+C_2(Q)]\mathcal{G}(z,q),$
$\mbox{where}\;C_2(Q)=\DF{C(Q)}{(Q-1)^{1/2}}.$\\
\item $\mathcal{G}(z,q)\le\mathcal{G}(z,q_0)+\mathcal{G}(q_0,q)\le\mathcal{G}(z,q_0)+\DF{C(Q)s_0}{(Q-1)^{1/2}}\le [1+C_2(Q)]\mathcal{G}(z,q_0)$.
Additionally, we define
$$0<s_0\le r\le s_1=K^{-1}s_2<s_2=2^{-1}\inf\{|p_i-p_j|:1\le i<j\le J\}.$$
Let $\Phi:{\bf Q}_Q(\mathbb{R}^n)\to{\bf Q}_Q(\mathbb{R}^n)$ be the
semi-retraction mapping constructed in Section 2.10 in
\cite{Frederick J. Almgren} corresponding to
$J,k_1,\hdots,k_J,p_1,\hdots,p_J,s_1,s_2$ above. For each $p\in{\bf
Q}_Q(\mathbb{R}^n)$
$$\mathcal{G}\left(\Phi(p),p\right)\le\mathcal{G}(q_0,p)=2\left[\mathcal{G}(q_0,p)-2^{-1}\mathcal{G}(q_0,p)\right]$$
so that
\item $\mathcal{G}(\Phi(p),p)\le2[\mathcal{G}\left(p,q_0)-s_1/2\right].$
\end{enumerate}
\end{Almgrenm}

\begin{Almgren5}\label{energy}
Suppose $f\in\mathcal{Y}_2(\partial U^m(0,1),\mathbb{R}^n),q\in V,0<r<\infty,$
and
$$A=\partial B^m(0,1)\cap\{x:|f(x)-q|>r\}$$
with $\mathcal{H}^{m-1}(A)\le m\alpha(m)/4.$ Then
\begin{multline*}
2^{-1-4/(m-1)}(2/\pi)^{2m-4+2/(m-1)}(m-1)\beta(m)^{-1}\int_A\left(|f(x)-q|-r\right)^2
d\mathcal{H}^{m-1}\\ \le
\left[\mathcal{H}^{m-1}(A)\right]^{2/(m-1)}\text{dir}(f;A).
\end{multline*}
\end{Almgren5}

We are aiming to prove the following theorem:
\begin{thm}\label{reg_almost}
Suppose $u\in\mathcal{Y}_2(U^m(0,1),{\bf Q}_Q(\mathbb{R}^n))$ is a
strictly defined  Dirichlet $(c,\alpha)$-almost minimizer such that
$\text{Dir}(u;U^m(0,1))>0$. Then
$$u\in
C_{\text{loc}}^{0,\sigma}(U^m(0,1),{\bf Q}_{Q}(\mathbb{R}^n))$$ for
some $0<\sigma<1$.
\end{thm}

We first prove an energy growth estimate for points with small
normalized energy. This estimate is divided into two parts, on
strong branch points (see definition below) and non-strong-branch points. Then we show that the energy
density is zero for every point in the domain, which completes the
interior regularity.

\begin{definition}
For a multiple-valued function $u$, we define the strong branch set to be
\begin{multline*}
B_u=\{x\in U^m(0,1):x\;\mbox{is a Lebesgue point of}\;\xi\circ u,\\
\xi^{-1}\circ \rho\circ AV_{r,x}(\xi\circ u)=Q[[b_r]], $ $\mbox{for
any small enough radius}\; r>0, b_r\in \mathbb{R}^n\}.
\end{multline*}
where $AV_{r,x}(\xi\circ u):=\dashint_{\partial U^m(x,r)} \xi\circ
u$. Obviously, $x\in B_u$ implies that $u(x)=Q[[y]]$ for some
$y\in\mathbb{R}^n.$
\end{definition}

\begin{lem}[Hybrid Inequality]
There is a positive constant $C$, depending only on $m,n,Q,c,\alpha$
such that if $0<\lambda<1$, $0<\rho\le 1$ and $u$ is a strictly
defined $(c,\alpha)-$almost minimizer, then
\begin{equation}
E_{\rho/2}(u)\le \lambda
E_\rho(u)+C\left[\rho^\alpha+\lambda^{-1}\rho^{-m}\int_{U^m(0,\rho)}|\xi\circ
u-\mu|^2 dx\right],
\end{equation}
for any constant vector $\mu\in\mathbb{R}^{PQ}$, where
$E_r(u)=r^{2-m}\text{Dir}(u;U^m(0,r)).$
\end{lem}
\begin{proof}
We first use Fubini's Theorem as in Section 2.3 in \cite{hkl}, to obtain a
radius $\rho/2\le r\le\rho$ so that $u|\partial U^m(0,r)\in
\partial \mathcal{Y}_2(\partial U^{m}(0,r),{\bf Q}_Q(\mathbb{R}^n))$,
\begin{equation}\label{fubini1}
\int_{\partial U^{m}(0,r)}\left|\nabla_{\mbox{tan}} u\right|^2
d\mathcal{H}^{m-1}\le 8\int_{U^{m}(0,\rho)}|Du|^2 dx,
\end{equation}
\begin{equation}\label{fubini2}
\int_{\partial U^{m}(0,r)}|\xi\circ u-\mu|^2 d\mathcal{H}^{m-1}\le 8
\int_{ U^{m}(0,\rho)}|\xi\circ u-\mu|^2 dx.
\end{equation}
Let $h: U^m(0,r)\rightarrow {\bf Q}_Q(\mathbb{R}^n)$ be Dirichlet
minimizing with boundary values $u|\partial U^{m}(0,r)$.
\begin{equation*}
\begin{split}
\int_{U^{m}(0,r)}|Dh|^2dx &=\int_{\partial U^{m}(0,r)} \left<\xi\circ h,\DF{\partial(\xi\circ h)}{\partial r}\right>d\mathcal{H}^{m-1}\\
&=\int_{\partial U^{m}(0,r)}\left<\xi\circ h-\mu,\DF{\partial(\xi\circ h)}{\partial r}\right>d\mathcal{H}^{m-1}\\
&\le \left[\int_{\partial U^{m}(0,r)} |\xi\circ h-\mu|^2
d\mathcal{H}^{m-1}\right]^{1/2}\left[\int_{\partial B^{m}(0,r)}
\left|\DF{\partial h}{\partial r}\right|^2 d \mathcal{H}^{
m-1}\right]^{1/2}\\
&=\left[\int_{\partial U^{m}(0,r)} |\xi\circ u-\mu|^2
d\mathcal{H}^{m-1}\right]^{1/2}\left[\int_{\partial U^{m}(0,r)}
\left|\DF{\partial h}{\partial r}\right|^2 d \mathcal{H}^{
m-1}\right]^{1/2}.
\end{split}
\end{equation*}
By Section 2.6 in \cite{Frederick J. Almgren},
$$\int_{\partial U^{m}(0,r)}\left|\DF{\partial h}{\partial r}\right|^2d\mathcal{H}^{m-1}\le \int_{\partial
U^{m}(0,r)}|\nabla_{\mbox{tan}}
h|^2d\mathcal{H}^{m-1}=\int_{\partial U^{m}(0,r)} |\nabla_{
\mbox{tan}} u|^2d\mathcal{H}^{m-1}.$$ Therefore,
\begin{equation*}
\int_{U^{m}(0,r)}|Dh|^2dx\le \left[\int_{\partial
U^{m}(0,r)}|\xi\circ u-\mu|^2d
\mathcal{H}^{m-1}\right]^{1/2}\left[\int_{\partial U^{m}(0,r)}|
\nabla_{\mbox{tan}} u|^2d\mathcal{H}^{m-1}\right]^{1/2}
\end{equation*}
and
\begin{equation*}
\begin{split}
E_{\rho/2}(u) &=(\rho/2)^{2-m}\text{Dir}(u;U^m(0,\rho/2))\\
&\le (\rho/2)^{2-m}\text{Dir}(u;U^m(0,r))\\
&\le (\rho/2)^{2-m}\left[\int_{
U^{m}(0,r)} |Dh|^2 dx+cr^{m-2+\alpha}\right]\\
&\le (\rho/2)^{2-m}\left[\int_{\partial
U^{m}(0,r)}|\nabla_{\mbox{tan}} u|^2
d\mathcal{H}^{m-1}\right]^{1/2}\left[\int_{\partial U^{m}(0,r)}|
\xi\circ u-\mu|^2 d\mathcal{H}^{m-1}\right]^{1/2}\\
&+2^{m-2}c\rho^\alpha.
\end{split}
\end{equation*}
Then we obtained the desired estimate by applying the inequality $ab\le \DF{1}{2}\delta
a^2+\DF{1}{2}\delta^{-1}b^2$, with $\delta=\DF{\lambda}{2^m}$ and using (\ref{fubini1})(\ref{fubini2}) as follows
\begin{equation*}
\begin{split}
E_{\rho/2}(u)&\le
\left(\DF{\rho}{2}\right)^{2-m}\left(\DF{1}{2}\delta \int_{\partial
U^{m}(0,r)}|\nabla_{\mbox{tan}} u|^2d\mathcal{H}^{m-1}+\DF{1}{2}\delta^{-1}
\int_{\partial U^{m}(0,r)} |\xi\circ u-\mu|^2d\mathcal{H}^{m-1}\right)\\
&+2^{m-2}c\rho^\alpha\\
&\le \left(\DF{\rho}{2}\right)^{2-m}\left(4\delta\int_{U^{m}(0,\rho)}|Du|^2dx+4\delta^{-1}\int_{U^{m}(0,\rho)} |\xi\circ u-\mu|^2dx\right)+2^{m-2}c\rho^\alpha\\
&=\lambda E_\rho(u)+4^m \lambda^{-1}\rho^{-m}\int_{U^{m}(0,\rho)}|\xi\circ u-\mu|^2 dx+2^{m-2}c\rho^\alpha\\
&\le\lambda
E_\rho(u)+C\left[\rho^\alpha+\lambda^{-1}\rho^{-m}\int_{U^{m}(0,\rho)}|\xi\circ
u-\mu|^2 dx\right]
\end{split}
\end{equation*}
where $C=\max\{2^{m-2}c,4^m\}$.
\end{proof}
Let us introduce the family of functions that we will "blow up":
$$\mathcal{F}=\left\{u\in\mathcal{Y}_2(U^m(0,1),{\bf Q}_Q(\mathbb{R}^n)): u\;\mbox{is}\;(c,\alpha)-\mbox{almost minimizing and}\;0\in B_u\right\}$$

\begin{lem}[Energy Improvement]
There are positive constants $\epsilon_0,r_0,\eta$ and $\theta<1$
so that, for any $(c,\alpha)-$almost minimizer $u\in\mathcal{F}$ with
$E_{r_0}(u)<\epsilon_0^2$, one has
\begin{equation}\label{improvement}
E_{\theta r}(u)\le \theta^{\omega_{2.13}}\max\left\{\eta
r^\alpha,E_r(u)\right\},\forall 0<r<r_0,
\end{equation}
where the constant $0<\omega_{2.13}<1$ is defined by Theorem 2.13 in \cite{Frederick J. Almgren}.
\end{lem}
\begin{proof}
If this theorem were false, then for any fixed positive
$\theta<1/2$, there would exist $(c,\alpha)-$almost minimizers
$u_i\in\mathcal{F}$ and $r_i\to 0$ for which
$$\epsilon_i^2:=E_{r_i}(u_i)\to 0$$
\begin{equation}\label{contradiction}
r_i^{-\alpha}E_{\theta r_i}(u_i)\to \infty
\end{equation}
as $i\to \infty$, but
\begin{equation}\label{choice}
E_{\theta r_i}(u_i)>\theta^{\omega_{2.13}}\epsilon_i^2
\end{equation}
for all $i$. The above relation (\ref{contradiction}) clearly implies that
$$\epsilon_i^{-2} r_i^\alpha\to 0,\;\mbox{as}\;i\to\infty.$$
We define the blowing-up sequence $v_i:U^{m}(0,1)\to{\bf
Q}_Q(\mathbb{R}^n)$ by
$$v_i(x)=\mu(\epsilon_i^{-1})_\sharp\circ \left[u_i(r_i x)-\xi^{-1}\circ\rho\circ AV_{r_i,0}(\xi\circ u_i)\right],$$
where the subtraction makes sense because $\xi^{-1}\circ\rho\circ AV_{r_i,0}(\xi\circ u_i)=Q[[b_i]]$ for some $b_i\in\mathbb{R}^n$.\\
The Dirichlet energy of $v_i$ is clearly uniformly bounded by 1. As for the $L^2$ norms, we have
\begin{equation*}
\begin{split}
&\int_{U^m(0,1)} \mathcal{G}^2\left(u_i(r_i x)-\xi^{-1}\circ\rho\circ AV_{r_i,0}(\xi\circ u_i),Q[[0]]\right) dx\\
&=\int_{U^m(0,1)} \mathcal{G}^2\left(u_i\circ \mu(r_i)(x)-\xi^{-1}
\circ\rho\circ AV_{1,0}(\xi\circ u_i\circ \mu(r_i)),Q[[0]]\right)dx\\
&\le \text{Lip}(\xi^{-1})^2\;\text{Lip}(\rho)^2\int_{U^m(0,1)}|\xi\circ u_i\circ \mu(r_i)(x)-AV_{1,0}(\xi\circ u_i\circ \mu(r_i))|^2 dx\\
&\le C\;\text{Lip}(\xi^{-1})^2\;\text{Lip}(\rho)^2\; \text{Dir}\left(\xi\circ u_i\circ\mu(r_i);U^m(0,1)\right)\\
&=C\;\text{Lip}(\xi^{-1})^2\;\text{Lip}(\rho)^2\; \epsilon_i^2,
\end{split}
\end{equation*}
where the second inequality comes from the Poincar\'{e} inequality (see Corollary 6.1 in \cite{zw}).\\
Using compactness theorem for multiple-valued functions (see Theorem
4.2 in \cite{zw1}), there is a subsequence of $v_i$ (still denoted
as $v_i$) such that $v_i$ converges weakly to
$v\in\mathcal{Y}_2(U^m(0,1),{\bf Q}_Q(\mathbb{R}^n))$. Moreover, by
similar argument as in Section 6.4 of \cite{zw}, we can show that this
convergence is actually
strong and $v$ is Dirichlet minimizing. \\
Let's estimate the extra term in the hybrid inequality
\begin{equation*}
\begin{split}
&\dashint_{U^{m}(0,rr_i)}\left|\xi\circ u_i-AV_{r,0}(\xi\circ u_i\circ \mu(r_i))\right|^2dx\\
&=\dashint_{U^{m}(0,r)}\left|\xi\circ u_i(r_i x)-AV_{r,0}(\xi\circ u_i\circ \mu(r_i))\right|^2 dx\\
&\le Cr^{2-m}\int_{ U^{m}(0,r)}|D(\xi\circ u_i\circ \mu(r_i))|^2
dx\;\quad (\mbox{by Poincar\'{e}
inequality})\\
&=C(\epsilon_i)^2r^{2-m}\int_{U^{m}(0,r)}|Dv_i|^2
dx\\
&\le C(\epsilon_i)^2r^{2-m}\int_{U^{m}(0,r)}|Dv|^2
dx\;\quad(\mbox{by strong convergence})\\
&\le C(\epsilon_i)^2r^{2-m}r^{m-2+2\omega_{2.13}}\text{Dir}(v;U^m(0,1))\quad (\text{interior regularity of Dirichlet minimizer})\\
&=C(\epsilon_i)^2 r^{2\omega_{2.13}}.
\end{split}
\end{equation*}
Applying the hybrid inequality to $u_i$ with $\rho=2\theta r_i$,
we get
\begin{equation*}
\begin{split}
E_{\theta r_i}(u_i)&\le \lambda E_{2\theta
r_i}(u_i)+C\left[\left(2\theta r_i\right)^\alpha+\lambda^{-1}
\dashint_{U^m(0,{2\theta
r_i})}\left|\xi\circ u_i-AV_{2\theta,0}(\xi\circ u_i\circ \mu(r_i))\right|^2 dx\right]\\
&\le \lambda E_{2\theta r_i}(u_i)+C\left[\left(2\theta
r_i\right)^\alpha+\lambda^{-1} C(\epsilon_i)^2(2\theta)^{2\omega_{2.13}}\right]\\
\end{split}
\end{equation*}
Choosing a positive integer $k=k(\theta)$ for which
$2^k\theta\le 1<2^{k+1}\theta$, we iterate $k-1$ more times to
obtain (we suppress all universal constant as $C$)
\begin{equation*}
\begin{split}
E_{\theta r_i}(u_i)&\le \lambda^k E_{2^k\theta
r_i}(u_i)+\sum_{j=1}^k \lambda^{j-1}C\left(2^j\theta r_i\right)^\alpha+\sum_{j=1}^k \lambda^{j-2}C(\epsilon_i)^2\left(2^j\theta\right)^{2\omega_{2.13}}\\
&\le \lambda^k 2^{m-2}\epsilon_i^2+\sum_{j=1}^\infty \lambda^{j-1}C\left(2^j\theta r_i\right)^\alpha+\sum_{j=1}^\infty \lambda^{j-2}C\left(\epsilon_i\right)^2(2^j\theta)^{2\omega_{2.13}}\\
&=\lambda^k 2^{m-2}\epsilon_i^2+C(\theta
r_i)^\alpha\DF{2^\alpha}{1-2^\alpha \lambda}+\DF{\lambda\cdot
2^{2\omega_{2.13}}}{1-\lambda\cdot
2^{2\omega_{2.13}}}C\lambda^{-2}\theta^{2\omega_{2.13}}(\epsilon_i)^2\\
&=\left[\lambda^k2^{m-2}+C\theta^\alpha \left(r_i^\alpha
\epsilon_i^{-2}\right) \DF{2^\alpha}{1-2^\alpha
\lambda}+\DF{\lambda\cdot 2^{2\omega_{2.13}}}{1-\lambda\cdot
2^{2\omega_{2.13}}}C\lambda^{-2}\theta^{2\omega_{2.13}}\right](\epsilon_i)^2
\end{split}
\end{equation*}
Taking $\lambda=\theta^{\DF{m+\omega_{2.13}}{k}}$, we have
\begin{equation*}
\lambda^k\cdot
2^{m-2}=\theta^{m+\omega_{2.13}}\cdot 2^{m-2}=\theta^m\cdot 2^{m-2}\cdot
\theta^{\omega_{2.13}}\le
(1/2)^m\cdot 2^{m-2}\theta^{\omega_{2.13}}\le \theta^{\omega_{2.13}}/4
\end{equation*}
Since $\lambda=\theta^{\DF{m+\omega_{2.13}}{k}}\le
(2^{-k})^{\DF{m+\omega_{2.13}}{k}}=2^{-(m+\omega_{2.13})}$,
\begin{equation*}
\begin{split}
\DF{\lambda\cdot 2^{2\omega_{2.13}}}{1-\lambda\cdot
2^{2\omega_{2.13}}}C\lambda^{-2}\theta^{2\omega_{2.13}}
&\le \DF{2^{2\omega_{2.13}}C}{1-2^{\omega_{2.13}-m}}\theta^{-\DF{m+\omega_{2.13}}{k}}\theta^{2\omega_{2.13}}\\
&\equiv M\theta^{\omega_{2.13}-\DF{m+\omega_{2.13}}{k}}\theta^{\omega_{2.13}},
\end{split}
\end{equation*}
where $M=\DF{2^{2\omega_{2.13}}C}{1-2^{\omega_{2.13}-m}}$.\\
Let's choose $\theta$ small enough such that
$\theta^{\omega_{2.13}-\DF{m+\omega_{2.13}}{k}}\le 1/4M$. This is possible because it
is equivalent to
$$\theta^{\omega_{2.13}}\le \theta^{\DF{m+\omega_{2.13}}{k}}/4M.$$
Notice that $\theta\ge 2^{-1-k}$, the right side of above one is
greater than
$$2^{-(k+1)(m+\omega_{2.13})/k}/4M$$
which is bounded from below although when $\theta$ goes to zero, $k$ goes to infinity.\\
Notice that $\epsilon_i^{-2} r_i^\alpha\to
0,\;\mbox{as}\;i\to\infty,$ for $i$ sufficiently large
enough,we have
$$E_{\theta r_i}(u_i)\le\left(\DF{1}{4}\theta^{\omega_{2.13}}+\DF{1}{4}\theta^{\omega_{2.13}}\right)\epsilon_i^2
<\theta^{\omega_{2.13}}\epsilon_i^2,$$ contradicting the choice of $u_i$ in (\ref{choice}).
\end{proof}
Iteration of (\ref{improvement}) as Section 3.5 of \cite{hkl} leads to the
following energy decay estimate.

\begin{thm}[Energy decay for strong branch points] If $u\in \mathcal{F}$ is $(c,\alpha)$-almost minimizing, with $r_0^{2-m}\int_{U^m(0,r_0)}|Du|^2\le
\epsilon_0^2$, then
$$r^{2-m}\int_{U^m(0,r)}|Du|^2\le Cr^\alpha,\;\mbox{for}\;0\le r\le r_0$$
where $\epsilon_0$ is as in the Energy Improvement.
\end{thm}

Now we turn to non-strong-branch points and are going to prove an energy decay estimate by induction on $Q$. In particular, without loss of generality, we assume
$$\xi^{-1}\circ\rho\circ AV_{1,0}(\xi\circ u)\not= Q[[y]]\;\mbox{for any}\;y\in\mathbb{R}^n.$$
Let $q^{*}=\rho\circ AV_{1,0}(\xi\circ
u),q=\xi^{-1}(q^{*})=\sum_{i=1}^Q[[q_i]],$ and $q_0=\sum_{i=1}^J
k_i[[p_i]]$ is obtained from $q$ using Section 2.9 in
\cite{Frederick J. Almgren}. By the argument in Section 3 of \cite{zw},
$J\ge 2$.
\begin{lem}[Construction of a comparison function]
Assume $m,n,Q\ge 2$ and $u\in\mathcal{Y}_2(U^m(0,1),{\bf
Q}_Q(\mathbb{R}^n))$ is a strictly defined $(c,\alpha)-$almost
minimizer. Let $\; t_Q>0$ be a real number such that
 $$\left[\DF{2\text{Lip}(\rho) \text{Lip}(\xi^{-1}) \left(1+C_2(Q)\right)}{s_0}\right]^2 Lip(\xi)^2 t_Q^{m-1}\le m\alpha(m)/4.$$
If $u$ satisfies $\text{dir}(u;\partial U^m(0,1))<t_Q^{m-1}$, then
there is a comparison function $g\in\mathcal{Y}_2(U^m(0,1),{\bf
Q}_Q(\mathbb{R}^n))$ satisfying the followings:
\begin{enumerate}
\item $g=u$ on $\partial U^m(0,1),$
\item $g|_{U^m(0,{1-t_Q})}=\sum_{i=1}^J g_i,$ where $g_i\in\mathcal{Y}_2(U^m(0,{1-t_Q}),{\bf Q}_{k_i}(\mathbb{R}^n))$ is Dirichlet minimizing and $\sum_{i=1}^J k_i=Q$,
\item $\text{Dir}(g;U^m(0,1)\sim U^m(0,{1-t_Q}))\le \delta_Q\;
\text{dir}(g;\partial U^m(0,1)),$
\end{enumerate}
for some constant $\delta_Q$.
\end{lem}
\begin{proof}
For each $x\in B^m(0,1)\sim B^m(0,{1-t_Q})$, we define $\tau:B^m(0,1)\sim B^m(0,{1-t_Q})\to\mathbb{R}$ and $F,G,H:B^m(0,1)\sim B^m(0,{1-t_Q})\to\mathbb{R}^{PQ}$ by
\begin{equation*}
\begin{split}
\tau(x)&=t_Q^{-1}(1-|x|),\\
F(x)&=\xi\circ u(x/|x|),\\
H(x)&=\xi\circ \Phi\circ u(x/|x|),\\
G(x)&=(1-\tau(x))F(x)+\tau(x)H(x).
\end{split}
\end{equation*}
Define $g:B^m(0,1)\sim B^m(0,{1-t_Q})\to{\bf Q}_Q(\mathbb{R}^n)$ by
$$g|B^m(0,1)\sim B^m(0,{1-t_Q})=\xi^{-1}\circ \rho\circ G|B^m(0,1)\sim B^m(0,{1-t_Q}).$$
On $\partial B^m(0,1)$, $g(x)=u(x)$. On $\partial B^m(0,{1-t_Q})$, $g(x)=\Phi\circ u(x/|x|)=\sum_{i=1}^J h_i$, where $h_i\in\partial\mathcal{Y}_2(\partial B^m(0,{1-t_Q}),{\bf Q}_{k_i}(\mathbb{R}^n))$, for each $i=1,\hdots,J$. \\Let $g_i\in\mathcal{Y}_2(U^m(0,{1-t_Q}),{\bf Q}_{k_i}(\mathbb{R}^n))$ be a Dirichlet minimizing function with boundary $h_i$ and $g|B^m(0,{1-t_Q})=\sum_{i=1}^J g_i$. This completes (1) and (2).\\
By the definition of $g$,
\begin{equation*}
\begin{split}
&\int_{U^m(0,1)\sim U^m(0,{1-t_Q})} |\nabla g|^2dx\\
&=\int_{U^m(0,1)\sim U^m(0,{1-t_Q})} \left|\nabla (\xi^{-1}\circ \rho\circ G(x))\right|^2dx\\
&\le \left[\text{Lip}(\xi^{-1})\text{Lip}(\rho)\right]^2\int_{U^m(0,1)\sim U^m(0,{1-t_Q})} \left|\nabla G(x)\right|^2dx\\
&=
\left[\text{Lip}(\xi^{-1})\text{Lip}(\rho)\right]^2\left[\int_{U^m(0,1)\sim
U^m(0,{1-t_Q})} \left|\DF{\partial G}{\partial
r}\right|^2dx+\int_{U^m(0,1)\sim U^m(0,{1-t_Q})}  |
\nabla_{\text{Tan}} G(x)|^2dx\right]
\end{split}
\end{equation*}
Also, by the definition of $G$, we can compute
$$\left|\DF{\partial G}{\partial r}\right|=\left|t_Q^{-1}(F(x)-H(x))\right|=\left|t_Q^{-1}\left(\xi\circ u(x/|x|)-\xi
\circ \Phi\circ u(x/|x|)\right)\right|,$$ and
\begin{equation*}
\begin{split}
\left|\nabla_{\partial U^m(0,{|x|})} G(x)\right|&=|x|^{-1}\left|\nabla_{\partial U^m(0,1)} G(x)\right|\\
&=\left|\DF{1-\tau(x)}{|x|}\nabla_{\partial U^m(0,1)}F(x)+\DF{\tau(x)}{|x|}\nabla_{\partial U^m(0,1)} H(x)\right|\\
&\le \DF{1}{1-t_Q}\left|\nabla_{\partial U^m(0,1)}(\xi\circ u(x/|x|))-\nabla_{\partial U^m(0,1)}(\xi\circ\Phi\circ u(x/|x|))\right|\\
&\le
\DF{1}{1-t_Q}\text{Lip}(\xi)\;\left(1+\text{Lip}(\Phi)\right)\left|\nabla_{\partial
U^m(0,1)}u(x/|x|)\right|.
\end{split}
\end{equation*}
Therefore,
\begin{equation}
\begin{split}
&\int_{U^m(0,1)\sim U^m(0,{1-t_Q})} |\nabla G(x)|^2dx \\
&\le t_Q^{-2}\int_{U^m(0,1)\sim U^m(0,{1-t_Q})} \left[\xi\circ
u(x/|x|)-\xi
\circ \Phi\circ u(x/|x|)\right]^2dx\\
&+(1-t_Q)^{-2}\left[\text{Lip}(\xi)(1+\text{Lip}(\Phi))\right]^2\int_{U^m(0,1)\sim U^m(0,{1-t_Q})} \left|\nabla_{\partial U^m(0,1)} (u(x/|x|))\right|^2dx\\
&\le t_Q^{-1}\int_{\partial U^m(0,1)} \left[\xi\circ u(x/|x|)-\xi
\circ \Phi\circ u(x/|x|)\right]^2d\mathcal{H}^{m-1}\\
&+\DF{t_Q}{(1-t_Q)^2}\left[\text{Lip}(\xi)(1+\text{Lip}(\Phi))\right]^2\text{dir}(u;\partial
U^m(0,1)).
\end{split}
\end{equation}
Now let us estimate
$$\int_{\partial U^m(0,1)} \left[\xi\circ u(x/|x|)-\xi
\circ \Phi\circ u(x/|x|)\right]^2d\mathcal{H}^{m-1}.$$ Let $Z=\partial
U^m(0,1)\cap\left\{x:\mathcal{G}(u(x),q_0)>s_1\right\}$
\begin{equation*}
\begin{split}
Z&\subset\partial U^m(0,1)\cap\left\{x:\mathcal{G}(u(x),q)>\DF{s_1}{1+C_2(Q)}\right\}(\mbox{by Theorem 2.12 in \cite{Frederick J. Almgren}})\\
&\subset\partial U^m(0,1)\cap\left\{x:|\xi\circ u(x)-q^*|>\DF{s_1}{\text{Lip}(\xi^{-1})(1+C_2(Q))}\right\}\\
&\subset\partial U^m(0,1)\cap\left\{x:|\xi\circ
u(x)-q^*|>\DF{s_0}{2\text{Lip}(\xi^{-1})(1+C_2(Q))}\right\}\equiv A.
\end{split}
\end{equation*}
Therefore,
\begin{equation}\label{estimate_of_a}
\begin{split}
&\mathcal{H}^{m-1}(A)\left(\DF{s_0}{2\text{Lip}(\xi^{-1})(1+C_2(Q))}\right)^2\\
&\le \int_{\partial U^m(0,1)} |\xi\circ u(x)-q^*|^2 d\mathcal{H}^{m-1}\\
&=\int_{\partial U^m(0,1)} |\xi\circ u(x)-\rho\circ AV_{1,0}(\xi\circ u)|^2 d\mathcal{H}^{m-1}\\
&\le \text{Lip}(\rho)^2 \int_{\partial U^m(0,1)} \left|\xi\circ
u(x)-AV_{1,0}(\xi\circ u)\right|^2 d\mathcal{H}^{m-1}
\end{split}
\end{equation}
and
\begin{equation*}
\begin{split}
\mathcal{H}^{m-1}(Z)&\le\mathcal{H}^{m-1}(A)\\
&\le  \left[\DF{2\text{Lip}(\rho)\cdot \text{Lip} (\xi^{-1})
(1+C_2(Q))}{s_0}\right]^2 \int_{\partial U^m(0,1)}
\left|\xi\circ u(x)-AV_{1,0}(\xi\circ u)\right|^2 d\mathcal{H}^{m-1}\\
&\le  \left[\DF{2\text{Lip}(\rho)\cdot \text{Lip} (\xi^{-1})
(1+C_2(Q))}{s_0}\right]^2 \text{Lip}(\xi)^2 \text{dir}(u;\partial
U^m(0,1)),
\end{split}
\end{equation*}
where the last inequality follows from \cite{Frederick J. Almgren} \S A.1.6(3).\\
This permits us to estimate
\begin{equation*}
\begin{split}
&\int_{\partial U^m(0,1)} \left|\xi\circ u-\xi\circ\Phi\circ u\right|^2 d\mathcal{H}^{m-1}\\
&=
\int_{Z} \left|\xi\circ u-\xi\circ\Phi\circ u\right|^2 d\mathcal{H}^{m-1}\\
&\le \text{Lip}(\xi)^2 \int_Z \left[\mathcal{G}(u,\Phi\circ u)\right]^2d\mathcal{H}^{m-1}\\
&\le 4 \text{Lip}(\xi)^2  \int_Z \left[\mathcal{G}(u(x),q_0)-\DF{s_1}{2}\right]^2d\mathcal{H}^{m-1}(\mbox{by Theorem 2.12(4) in \cite{Frederick J. Almgren}})\\
&\le 4 \text{Lip}(\xi)^2  \int_Z \left[\left(1+C_2(Q)\right)\;\mathcal{G}(u(x),q)-\DF{s_1}{2}\right]^2d\mathcal{H}^{m-1}\\
&(\mbox{because on}\;Z ,\mathcal{G}(u(x),q_0)>s_1\ge s_0\;\mbox{and Theorem 2.12(2) in \cite{Frederick J. Almgren}})\\
&\le 4 \left[\text{Lip}(\xi)\text{Lip}(\xi^{-1})\left(1+C_2(Q)\right)\right]^2  \int_Z \left[\left|\xi\circ u(x)-q^{*}\right|-\DF{s_1}{2\text{Lip}(\xi^{-1})\left(1+C_2(Q)\right)}\right]^2d\mathcal{H}^{m-1}\\
&\le 4 \left[\text{Lip}(\xi)\text{Lip}(\xi^{-1})(1+C_2(Q))\right]^2  \int_Z \left[\left|\xi\circ u(x)-q^{*}\right|-\DF{s_0}{2\text{Lip}(\xi^{-1})\left(1+C_2(Q)\right)}\right]^2d\mathcal{H}^{m-1}\\
&(\mbox{because}\;s_0\le s_1)\\
&\le 4 \left[\text{Lip}(\xi)\text{Lip}(\xi^{-1})(1+C_2(Q))\right]^2 C_1\left[\mathcal{H}^{m-1}(A)\right]^{2/(m-1)}\text{dir}(\xi\circ u;A)\\
&(\mbox{by Theorem A.1.6(17) in \cite{Frederick J. Almgren} and the estimate of}\;\mathcal{H}^{m-1}(A)\; \text{in}\; (\ref{estimate_of_a}))\\
&\left(C_1=[2^{-1-4/(m-1)}(2/\pi)^{2m-4+2/(m-1)}(m-1)\beta(m)^{-1}]^{-1}\right)\\
&\le C\left(\xi,\Phi,m,Q\right)\text{dir}\left(u;\partial
U^m(0,1)\right)^{1+\DF{2}{m-1}}
\end{split}
\end{equation*}
In summary, we have
\begin{equation*}
\begin{split}
&\int_{U^m(0,1)\sim U^m(0,{1-t_Q})} \left|\nabla g\right|^2dx\\&\le
C(\xi,\Phi,m,n,Q)\left[t_Q^{-1}\text{dir}(g;\partial U^m(0,1))^{1+\DF{2}{m-1}}+t_Q\text{dir}(g;\partial U^m(0,1))\right]\\
&=C(\xi,\Phi,m,n,Q)\left[t_Q^{-1}\text{dir}(g;\partial U^m(0,1))^{\DF{2}{m-1}}+t_Q\right]\text {dir}(g;\partial U^m(0,1))\\
&\le2C(\xi,\Phi,m,n,Q)\;t_Q\;\text{dir}(g;\partial U^m(0,1))\\
&=\delta_Q\;\text{dir}(g;\partial U^m(0,1)),
\end{split}
\end{equation*}
where $\delta_Q=2C(\xi,\Phi,m,n,Q) t_Q$.
\end{proof}

\begin{rem*}
\begin{enumerate}
\item Here is the scaled version:
Assume $(c,\alpha)-$almost minimizer $u\in\mathcal{Y}_2(U^m(0,r),{\bf Q}_Q(\mathbb{R}^n))$ satisfies
$$\text{dir}(u;\partial U^m(0,r))< r^{m-3}t_Q^{m-1},$$
then the constructed comparison function $g$ with $g=u$ on $\partial
U^m(0,r)$ satisfies
$$\text{Dir}(g;U^m(0,r)\sim U^m(0,{r(1-t_Q)}))\le \delta_Q r\text{dir}(u;\partial U^m(0,r)).$$
\item The smallness assumption of $\text{dir}(u;\partial U^m(0,r))$
can be replaced by the smallness of $\text{Dir}(u;U^m(0,r))$. This
is because we choose $g\in\mathcal{Y}_2(U^m(0,r),{\bf
Q}_Q(\mathbb{R}^n))$ with $g=u$ on $\partial U^m(0,r)$. By the
squeeze formula (\cite{Frederick J. Almgren} \S 2.6) , we have
\begin{equation*}
\begin{split}
2r\text{dir}(u;\partial U^m(0,r))&=2r\text{dir}(g;\partial U^m(0,r))\\
&=(m-2)\text{Dir}(g;U^m(0,r))+r\text{Dir}(g;U^m(0,r))'\\
&\le (m-2)\text{Dir}(u;U^m(0,r))+r\text{Dir}(g;U^m(0,r))'
\end{split}
\end{equation*}
So in the blowing-up analysis, by choosing a good slicing by $\partial U^m(0,r)$ and rescaling, we can get small $\text{Dir}(g;U^m(0,r))'$ and hence the smallness of $\text{dir}(u;\partial U^m(0,r)).$
\end{enumerate}

\end{rem*}
\begin{thm}[Energy decay for non-strong-branch points]
Let $m,n\ge 2$. There exists a small number $\epsilon=\epsilon(Q,m,n,c,\alpha)$ such that any $(c,\alpha)-$almost minimizer $u\in\mathcal{Y}_2(U^m(0,1),{\bf Q}_Q(\mathbb{R}^n))$ with $0\notin B_v$ such that $r_0^{2-m}\text{Dir}(u;U^m(0,r_0))<\epsilon$ for some $0<r_0\le 1$ satisfies
$$\text{Dir}(u;U^m(0,r))\le Cr^{m-2+\beta},0<r\le r_0$$
for some universal positive constants $C,\beta$.
\end{thm}
\begin{proof}
Let $g$ be the comparison function on $B^m(0,r)$,
\begin{equation*}
\begin{split}
\text{Dir}(u;U^m(0,r))&\le \text{Dir}(g,U^m(0,{r(1-t_Q)}))+\delta_Q r\text{dir}(g;\partial U^m(0,r))+cr^{m-2+\alpha}\\
&=\text{Dir}(g;U^m(0,{r(1-t_Q)}))+\delta_Q r\text{dir}(u;\partial U^m(0,r))+cr^{m-2+\alpha}\\
&\le \text{Dir}(g,U^m(0,{r(1-t_Q)}))+\delta_Q r \text{Dir}(u,U^m(0,r))'+cr^{m-2+\alpha}
\end{split}
\end{equation*}
Applying the interior regularity of Dirichlet minimizing function to $g|B^m(0,{r(1-t_Q)})$,
$$\text{Dir}(u;U^m(0,r))\le r^{m-2+2\omega_{2.13}}+\delta_Q  r \text{Dir}(u;U^m(0,r))'+cr^{m-2+\alpha}.$$
Denote $\omega=\min\{2\omega_{2.13},\alpha\},N=1+c, D(r)=\text{Dir}(u;U^m(0,r))$, we have
$$D(r)\le \delta_Q r D'(r)+Nr^{m-2+\omega},$$
i.e.
$$D'-\DF{D}{\delta_Q r}+\DF{N}{\delta_Q}r^{m-3+\omega}\ge 0.$$
Multiplying $r^{\DF{-1}{\delta_Q}}$,
$$\DF{d}{dr}\left[D(r)r^{\DF{-1}{\delta_Q}}+\DF{N}{\delta_Q\left(m-2+\omega-\DF{1}{\delta_Q}\right)}r^{m-2+\omega-\DF{1}{\delta_Q}}\right]\ge 0.$$
Therefore,
$$D(r)r^{\DF{-1}{\delta_Q}}+\DF{N}{\delta_Q\left(m-2+\omega-\DF{1}{\delta_Q}\right)}r^{m-2+\omega-\DF{1}{\delta_Q}}\le M:=D(1)+\DF{N}{\delta_Q\left(m-2+\omega-\DF{1}{\delta_Q}\right)}$$
$$D(r)\le Mr^{\DF{1}{\delta_Q}}+\DF{N}{1-\delta_Q(m-2+\omega)}r^{m-2+\omega}$$
Letting
$C=\max\left\{M,\DF{N}{1-\delta_Q(m-2+\omega)}\right\},\beta=\min\left\{\DF{1}{\delta_Q}-(m-2),\omega\right\}$
(which is positive since we can choose $\delta_Q$ small enough)
finishes the proof.
\end{proof}

In spirit of Morrey's growth lemma, it only remains to show Theorem
\ref{density} to prove Theorem \ref{reg_almost}.

\begin{thm}\label{density}
For any $(c,\alpha)-$almost minimizer $u$,
$$\lim\inf_{r\downarrow 0} r^{2-m}\text{Dir}(u;U^m(x,r))=0,\forall x\in U^m(0,1).$$
\end{thm}
\begin{rem*}
It is well-known (e.g \cite{ly} lemma 2.1.1) that the above one holds $\mathcal{H}^{m-2}$ a.e. Also, it suffices to prove this for $x=0$.
\end{rem*}
\begin{lem}[Monotonicity Formula]
If $u$ is a $(c,\alpha)-$almost minimizer, then for any $0<t<s<1$,
we have
$$\int_t^s r^{2-m}\int_{\partial U^{m}(0,r)} \left|\DF{\partial u}{\partial
r}\right|^2 d\mathcal{H}^{m-1}
dr-\DF{(m-2)c}{\alpha}(s^\alpha-t^\alpha)$$
$$\le s^{2-m}\text{Dir}(u;U^m(0,s))-t^{2-m}\text{Dir}(u;U^m(0,t)).$$ In particular,
$r^{2-m}\text{Dir}(u;U^m(0,r))+\DF{(m-2)c}{\alpha} r^\alpha$ is nondecreasing in
$r$. Hence the limiting density
$$\Theta_u(0):=\lim_{r\downarrow 0}
r^{2-m}\int_{U^{m}(0,r)}|Du|^2 dx$$ exists.
\end{lem}
\begin{proof}
For any $0<r<1$, define a function
$g:B^m(0,1)\to{\bf Q}_Q(\mathbb{R}^n)$ by setting
\begin{equation*}
g(x)=
\begin{cases} f\left(r\DF{x}{|x|}\right) & |x|\le r\\
f(x) & \mbox{otherwise}.
\end{cases}
\end{equation*}
By definition of almost minimizer, we have
\begin{equation*}
\begin{split}
&\int_{U^{m}(0,r)}|Du|^2 dx\\
&\le
cr^{m-2+\alpha}+\int_{U^{m}(0,r)}|Dg|^2 dx\\
&=cr^{m-2+\alpha}+\DF{r}{m-2}\int_{\partial
U^{m}(0,r)}\left|\nabla_{\text{tan}}u\right|^2
d\mathcal{H}^{m-1}\\
&=cr^{m-2+\alpha}+\DF{r}{m-2}\left[\DF{d}{d\rho}|_{\rho=r}\int_{U^m(0,\rho)}
|Du|^2 dx-\int_{\partial U^{m}(0,r)}\left|\DF{\partial
u}{\partial r}\right|^2 d\mathcal{H}^{m-1}\right]\\
\end{split}
\end{equation*}
Multiplying by $(m-2)r^{1-m}$ to both sides,
$$r^{2-m}\int_{\partial U^{m}(0,r)}\left|\DF{\partial
u}{\partial r}\right|^2 d\mathcal{H}^{m-1}-(m-2)cr^{\alpha-1}\le
\DF{d}{d\rho}|_{\rho=r}(\rho^{2-m}\text{Dir}(u;U^m(0,\rho))).$$ Then
integrate this from $t$ to $s$.
\end{proof}
From the monotonicity formula, following the standard technique (e.g. \cite{sl}), we obtain a tangent map $v\in\mathcal{Y}_2(B^m(0,1),{\bf Q}_Q(\mathbb{R}^n))$ of $u$ at the origin.
$v$ is a homogeneous degree zero Dirichlet minimizing function with
$$\theta_v(0)=\theta_u(0).$$
From the interior regularity, $v$ has to be constant hence
$\theta_u(0)=\theta_v(0)=0$. This finishes the proof of Theorem
\ref{density} and also the proof of Theorem \ref{reg_almost}.

\subsection{Branch set of almost minimizer}

We will now produce a Dirichlet almost minimizer with a fractal
branch set. We begin with some definitions:

\begin{enumerate}
\item A {\it losange} above an interval $I=[a,b]\subset\mathbb{R}$
is a multiple-valued function $u:[a,b]\rightarrow {\bf
Q}_2(\mathbb{R})$ whose the graph is a parallelogram with the
following vertices $[a,0], [b,0], [(a+b)/2, (b-a)/2],
[(a+b)/2,(a-b)/2]$ .
\item  A {\it pluri-losange} is a
multiple-valued function $u:[0,1]\rightarrow {\bf Q}_2(\mathbb{R})$
which admits a partition of $[0,1]$ in intervals $\{I_j\}$ such that
$u$ is a losange above $I_j$ or a map of the form $x\rightarrow 2[[
0]]$ above $I_j$ for each $j$.
\end{enumerate}

Figure 2 gives examples of pluri-losanges. It is clear that a
pluri-losange $u$ is Lipschitz with Lip$(u)\leq \sqrt{2}$ and one
readily checks that a pluri-losange is a Dirichlet $(2,1/2)-$almost
minimizer. We also notice that a pluri-losange is not a Dirichlet
quasiminimizer or an $\omega$-minimizer.

We will now create a particular sequence of pluri-losanges
$\{u^i\}$. We define $u^1:[0,1]\rightarrow{\bf Q}_2(\mathbb{R})$ by
requiring that
\begin{itemize}
\item[-] $u^1$ is a pluri-losange, \item[-] $u^1(1/3)=0$, \item[-] $u^1$ is a losange only
above $[1/3,2/3]$.
\end{itemize}
We define $u^2:[0,1]\rightarrow{\bf Q}_2(\mathbb{R})$ by requiring
that
\begin{itemize}
\item[-] $u^2$ is a pluri-losange, \item[-]
$u^2(1/3)=0$, \item[-] $u^2$ is a losange only above the intervals
$[1/3,2/3],[1/9,2/9]$ and $[7/9,8/9]$.
\end{itemize}

\begin{figure}[htbp]\label{cantor}
    \begin{center}
    \includegraphics[height=4cm,width=4cm]{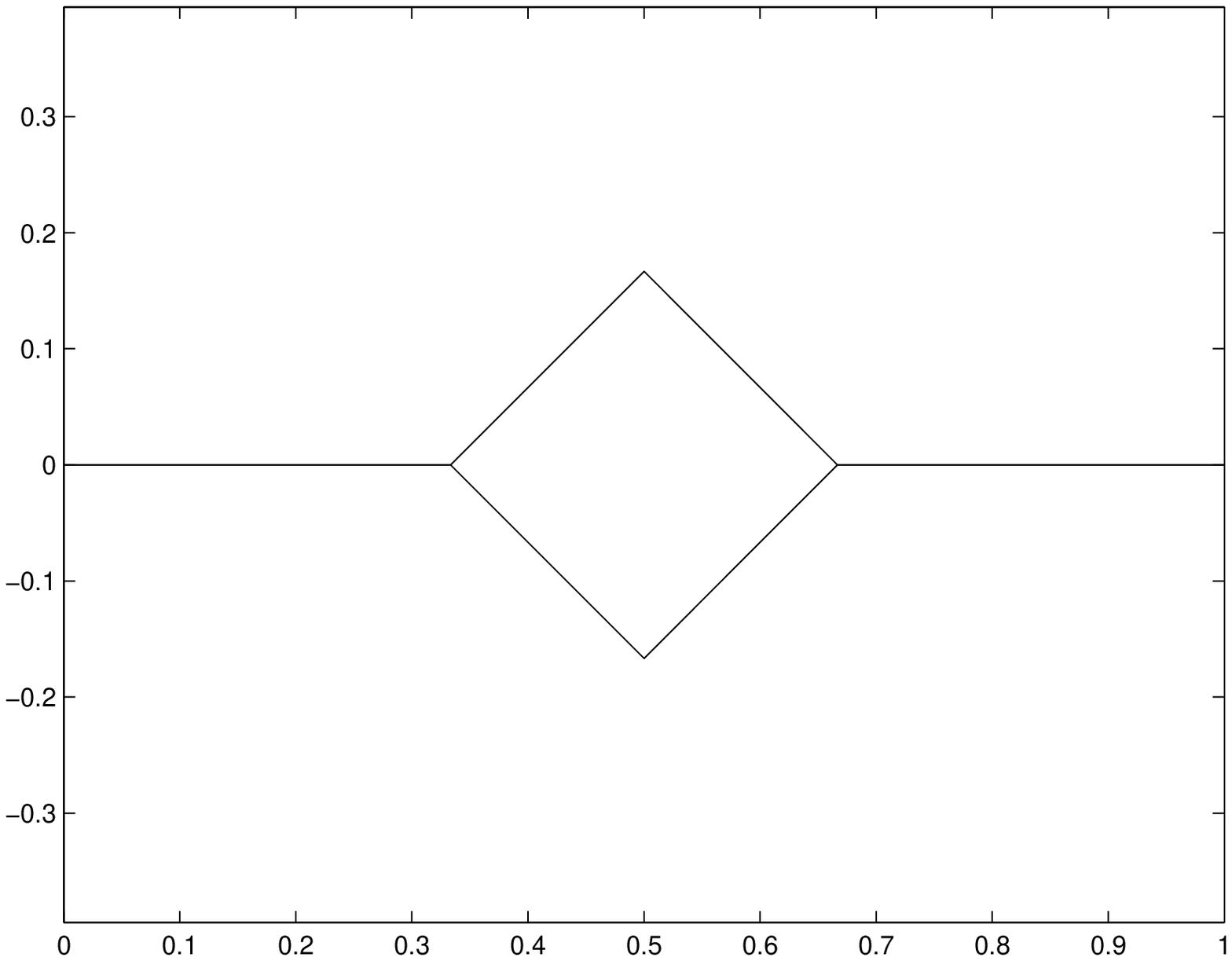}
    \includegraphics[height=4cm,width=4cm]{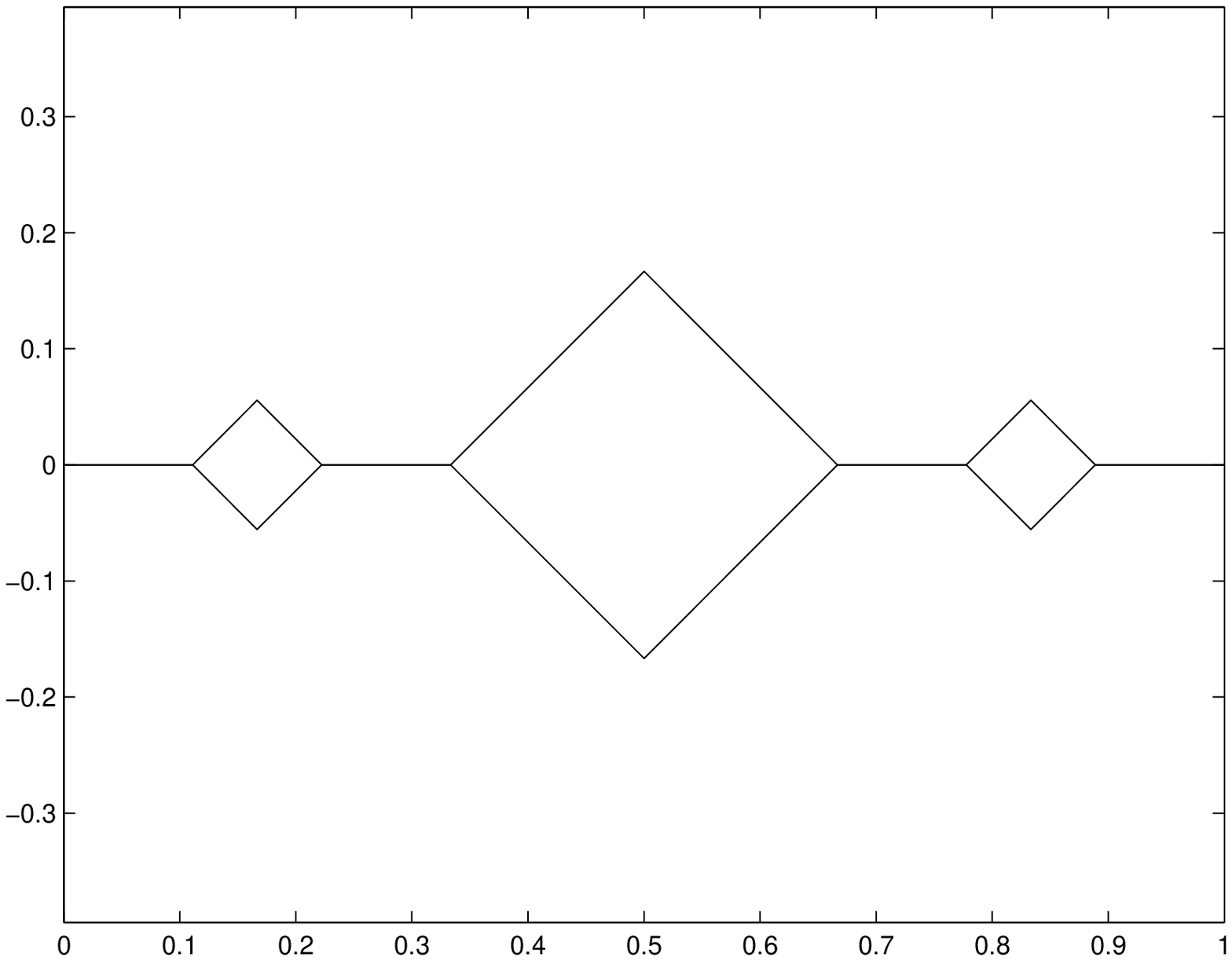}
    \includegraphics[height=4cm,width=4cm]{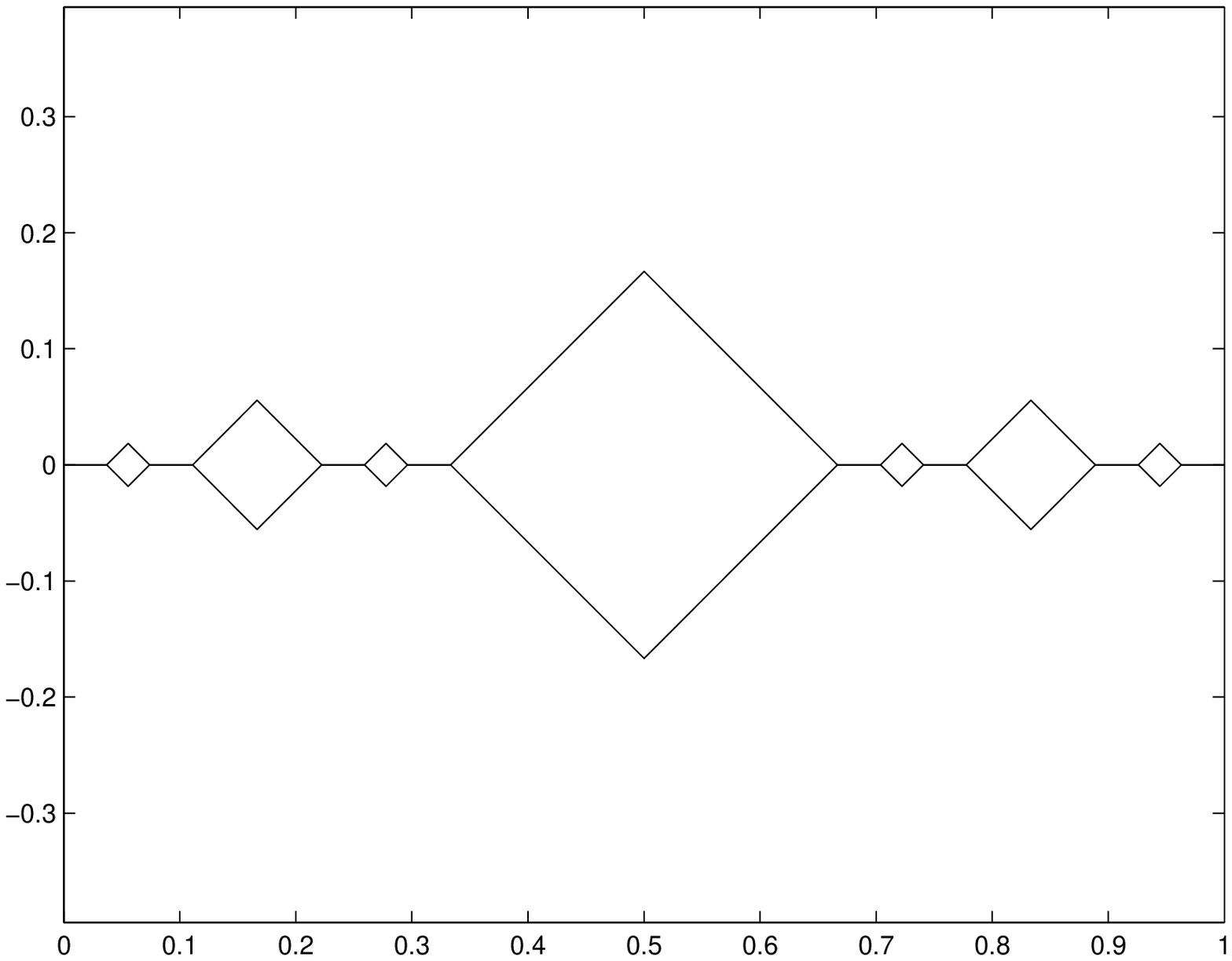}
    \caption{The pluri-losanges $u^1$, $u^2$ and $u^3$.}
    \end{center}
\end{figure}
The reader will easily imagine how the remaining terms of the
sequence are defined.
 By mimicking the arguments used in Section
\ref{quasiex}, we can extract a subsequence which converges
uniformly to a Lipschitz $(2,1/2)-$almost minimizer
$u:[0,1]\rightarrow {\bf Q}_2(\mathbb{R})$ whose the branch set is
the Cantor ternary set. This construction can also be adapted in
order to obtain an almost minimizer with a fat Cantor branch set.

\section{Acknowledgements}
The authors are thankful to Thierry De Pauw and Robert Hardt for
suggesting the problem and helpful discussions.

\end{document}